\documentclass[12pt,reqno]{amsart}
\usepackage{amssymb}
\usepackage{geometry}
\usepackage[latin1]{inputenc}
\usepackage[italian,english]{babel}
\usepackage{amsmath, amsfonts, amsthm}
\usepackage{graphicx}
\usepackage{pgfplots}
\usepackage{paralist}
\usepackage{mathtools, mathabx,bigints}
\numberwithin{equation}{section}
\usepackage{diagbox}
\usepackage{booktabs} 
\usepackage{comment} 
\usepackage{enumitem}
\usepackage{subfig}
\usepackage{bm}
\usepackage{url}
\usepackage{lineno}
\usepackage{hyperref}
\usepackage{cleveref}

\newtheorem{theorem}{Theorem}[section]
\newtheorem{lemma}[theorem]{Lemma}

\newtheorem{proposition}[theorem]{Proposition}

\newtheorem{definition}[theorem]{Definition}
\theoremstyle{definition}
\newtheorem{remark}[theorem]{Remark}
\theoremstyle{remark}

\newcommand{\Aps}{A^{\dag}}
\newcommand{\Aout}{A^{\ast}}
\newcommand{\norm}[1]{\left\Vert#1\right\Vert}

\newcommand{\abs}[1]{\left\vert#1\right\vert}

\newcommand{\R}{\mathbb{R}}
\newcommand{\N}{\mathbb{N}}

\newcommand{\argmin}{\arg\min}

\DeclareMathOperator{\dist}{dist}

\DeclareMathOperator\outi{out}
\newcommand*\interior[1]{\mathring{#1}}
{\left\{\begin{array}{@{}l@{}}}{\end{array}\right.}
\patchcmd{\abstract}{\scshape\abstractname}{\textbf{\abstractname}}{}{}
\makeatletter 
\def\@makefnmark{} 
\makeatother 

\makeatletter
\newcommand{\proofpart}[2]{%
  \par
  \addvspace{\medskipamount}%
  \noindent\emph{Case #1: #2}\par\nobreak
  \addvspace{\smallskipamount}%
  \@afterheading
}
\makeatother
\pgfplotsset{compat=1.18} 
\begin{document}

\title[The inverse obstacle problem for nonlinear inclusions]{The inverse obstacle problem for nonlinear inclusions}

\author[Mottola]{Vincenzo Mottola}
\address{Dipartimento di Ingegneria Elettrica e dell'Informazione \lq\lq M. Scarano\rq\rq, Universit\`a degli Studi di Cassino e del Lazio Meridionale, Via G. Di Biasio n. 43, 03043 Cassino (FR), Italy.}
\email{vincenzo.mottola@unicas.it}

\author[Corbo Esposito]{Antonio Corbo Esposito}
\address{Dipartimento di Ingegneria Elettrica e dell'Informazione \lq\lq M. Scarano\rq\rq, Universit\`a degli Studi di Cassino e del Lazio Meridionale, Via G. Di Biasio n. 43, 03043 Cassino (FR), Italy.}
\email{corbo@unicas.it}

\author[Faella]{Luisa Faella}
\address{Dipartimento di Ingegneria Elettrica e dell'Informazione \lq\lq M. Scarano\rq\rq, Universit\`a degli Studi di Cassino e del Lazio Meridionale, Via G. Di Biasio n. 43, 03043 Cassino (FR), Italy.}
\email{l.faella@unicas.it}

\author[Piscitelli]{Gianpaolo Piscitelli}
\address{Dipartimento di Scienze Economiche Giuridiche Informatiche e Motorie, Universit\`a degli Studi di Napoli Parthenope, Via Guglielmo Pepe, Rione Gescal, 80035 Nola (NA), Italy.}
\email{gianpaolo.piscitelli@uniparthenope.it (corresponding author)}

\author[Prakash]{Ravi Prakash}
\address{Departamento de Matem\'atica, Facultad de Ciencias F\'isicas y Matem\'aticas, Universidad de Concepci\'on, Avenida Esteban Iturra s/n, Bairro Universitario, Casilla 160 C, Concepci\'on, Chile.}
\email{rprakash@udec.cl}

\author[Tamburrino]{Antonello Tamburrino}
\address{Dipartimento di Ingegneria Elettrica e dell'Informazione \lq\lq M. Scarano\rq\rq, Universit\`a degli Studi di Cassino e del Lazio Meridionale, Via G. Di Biasio n. 43, 03043 Cassino (FR), Italy - and - Department of Electrical and Computer Engineering, Michigan State University, East Lansing, MI-48824, USA.}
\email{antonello.tamburrino@unicas.it.}

\setcounter{tocdepth}{1}

\begin{abstract}
The Monotonocity Principle (MP), stating a monotonic relationship between a material property and a proper corresponding boundary operator, is attracting great interest in the field of inverse problems, because of its fundamental role in developing real time imaging methods. Moreover, under quite general assumptions, a MP for elliptic PDEs with nonlinear coefficients has been established. This MP provided the basis for introducing a new imaging method to deal with the inverse obstacle problem, in the presence of nonlinear anomalies. This constitutes a relevant novelty because there is a general lack of quantitative and physic based imaging method, when nonlinearities are present.

The introduction of a MP based imaging method poses a set of fundamental questions regarding the performance of the method in the presence of noise.

The main contribution of this work is focused on theoretical aspects and consists in proving that (i) the imaging method is stable and robust with respect to the noise, (ii) the reconstruction approaches monotonically to a well-defined limit, as the noise level approaches to zero, and that (iii) the limit contains the unknown set and is contained in the outer boundary of the unknown set.

Results (i) and (ii) come directly from the Monotonicity Principle, while results (iii) requires to prove the so-called Converse of the Monotonicity Principle, a theoretical results of fundamental relevance to evaluate the ideal (noise-free) performances of the imaging method. 

The results are provided in a quite general setting for Calder\'on problem, and proved for three wide classes where the nonlinearity of the anomaly can be either bounded from infinity and zero, or bounded from zero only, or bounded by infinity only. These classes of constitutive relationships cover the wide majority of cases encountered in applications.

\vspace{0.2cm}
\end{abstract}

\maketitle
\section{Introduction}\label{sec:int}
This paper treats the inverse obstacle problem for elliptic PDEs in the presence of anomalies (the obstacles) described by a nonlinear constitutive relationship and a background described by a linear constitutive relationship (see \Cref{fig:geo}).

\begin{figure}[htp]
    \centering
    \includegraphics[width=0.4\textwidth]{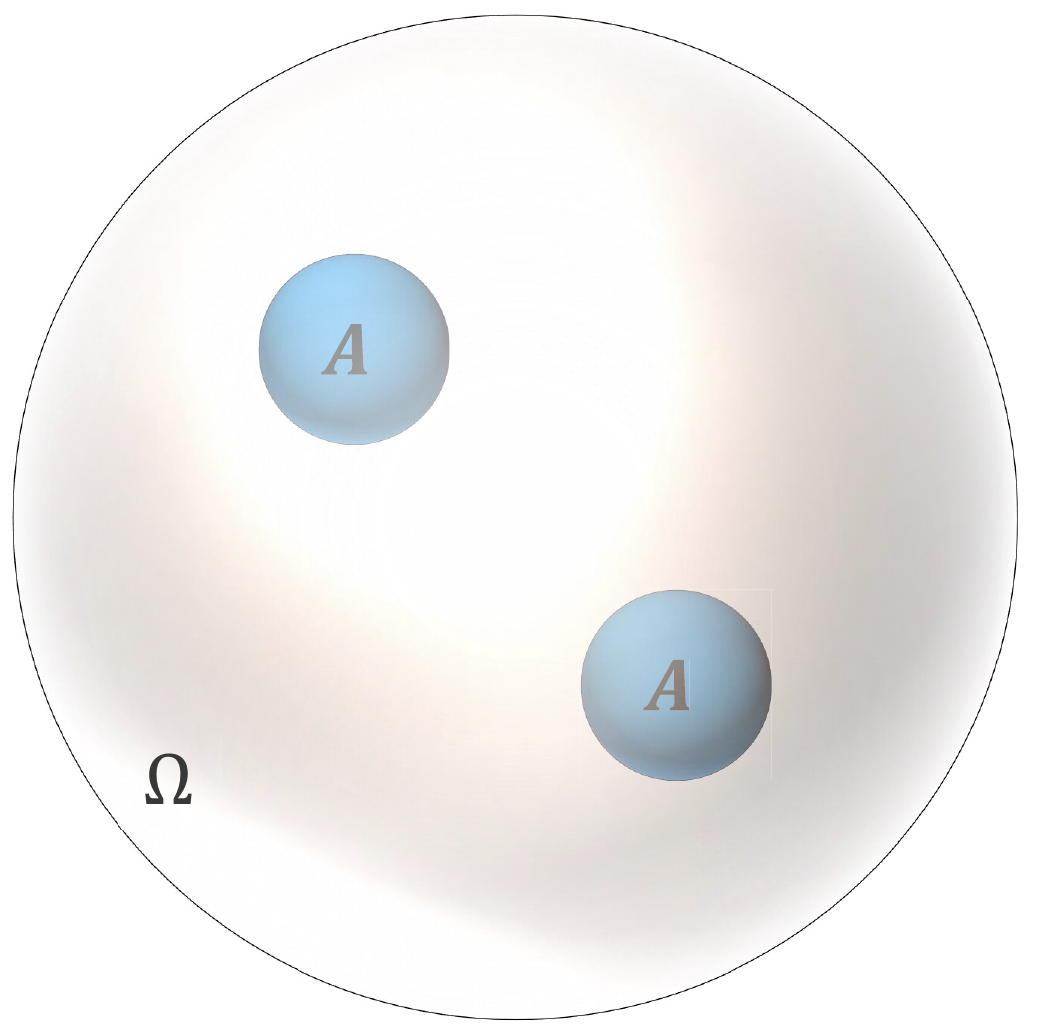}
    \caption{Description of the geometry of the problem: $\Omega$ is an open and bounded domain containing two inclusions (region $A$). The inclusions consist of a nonlinear material. The background consists of a linear material.}
    \label{fig:geo}
\end{figure}

Materials modeled by a nonlinear constitutive relationship are widely spread in different fields. With reference to nonlinear electrical conductivities, superconductive materials represent one of the most relevant examples since they are employed in several different applications such as high-energy storage systems, low-resistance energy transmission systems, and in nuclear fusion experimental facilities (superconductive magnets)~\cite{book:216193,book:1324162}. Nonlinear conductors can also be found in the termination of high voltage cables, where nonlinearity in electrical conductivities allows to effectively control the electric stress and mitigate the occurrence of partial discharges ~\cite{art:Bu08,art:Me18}. It is relevant that human tissues can also exhibit nonlinear electrical conductivity, as is the case during electroporation~\cite{art:tisnl} or for modeling the skin \cite{art:skin_nl}.

Magnetic materials, such as electrical steel or permanent magnets, are characterized by nonlinear magnetic permeability.
One of the most significant examples of application of nonlinear magnetic materials regards electrical machines such as transformers, electrical motors, and electric generators. Other relevant areas of application refer to surveillance and security, as for the detection of magnetic materials in boxes or containers~\cite{book:Dorn18, art:Mar15}, or to the nondestructive inspection of reinforcing bars in concrete \cite{art:Ig03,art:So05}.

Nonlinear dielectrics, as ferroelectric materials, are encountered in manufacturing tunable capacitors~\cite{art:die_nl}. Nonlinear dielectrics are also found in Schottky junctions~\cite{art:diode_nl}.

Composites provide infinite degrees of freedom in material design and their use is growing exponentially with time. A composite material is made by a matrix that embeds a filler. When the filler is nonlinear the overall behavior of the composite is nonlinear. An example is given by ferromagnetic composites in which ferromagnetic powders are embedded in a polymer matrix to accurately design the overall magnetic properties \cite{NOWOSIELSKI2015179}.

The nonlinear inverse obstacle problem treated in this contribution consists in reconstructing the shape, position and dimension of the anomaly $A$, by means of 
measurements carried out on $\partial \Omega$, the boundary of the domain of interest $\Omega$. 
Both the linear material of the background and nonlinear material of the inclusion are known.

The governing equations for the underlying (nonlinear PDE) problem are
\begin{equation}\label{P1}
\begin{cases}
\nabla\cdot\left(\gamma_{BG}(x)\nabla u_{A}(x)\right)=0 & \text{in $\Omega\setminus A$} \\
\nabla\cdot\left(\gamma_{NL}\left(x,\abs{\nabla u_A(x)}\right)\nabla u_A(x)\right)=0 & \text{in $A$} \\
u_{A}(x)=f(x) & \text{on $\partial\Omega$} \\
\gamma_{BG}(x)\partial_n u_{A}(x^+)=\gamma_{NL}\left(x,\abs{\nabla u_A(x^-)}\right) \partial_n u_A(x^-) & \text{on $\partial A$} \\
u_{A}(x^+)=u_A(x^-) & \text{on $\partial A$},
\end{cases}
\end{equation}
where $\Omega\subseteq \mathbb{R}^n$, $n\geq 2$, is a given open and bounded domain, $A\Subset \Omega$ is the region occupied by the nonlinear anomalies (see \Cref{fig:geo}), $u_A$ is the solution in the presence of the anomaly $A$, $\gamma_{BG}$ is the linear material property for the background, $\gamma_{NL}$ is the nonlinear material property for the anomaly.
Solutions $u_A$ and the applied boundary data $f$ belong to proper abstract
spaces (see Section \ref{sec:def} for details).
$u_A(x^+)$, $u_A(x^-)$, $\partial_n u_{A}(x^+)$ and $\partial_n u_{A}(x^-)$ are meant as the limits in $x$ evaluated from the outer $(+)$ or from the inner $(-)$ side of the boundary of $A$.

\Cref{P1} gives the governing equations for a variety of different physical phenomena. In the field of electromagnetism, it is a model for (i) steady-state conduction, where $\gamma=\sigma$ is the nonlinear electrical conductivity, (ii) magnetostatic, where $\gamma=\mu$ is the nonlinear magnetic permeability, and (iii) electrostatic, where $\gamma=\varepsilon$ is the nonlinear electric permittivity.

The mathematical model of \eqref{P1} is relevant because it models different tomographic inspection techniques, such as Electrical Resistance Tomography (ERT), Magnetic Inductance Tomography (MIT), and Electrical Capacitance Tomography (ECT). ERT is applied for biomedical imaging with nonionizing radiations \cite{Ho94} and  industrial process tomography \cite{Ma09}, for example. MIT is used in the detection of magnetic materials in boxes or containers~\cite{book:Dorn18,art:Mar15} or in the inspection of concrete~\cite{art:Ig03,art:So05}, for example. ECT plays an important role in industrial process tomography for imaging multi-phase flow in pipes \cite{7879196}.

Hence, the development of real-time imaging algorithms is of great interest in electromagnetic tomography and applications. Despite this, very few of them are available for various implementations.

Imaging approaches based on Monotonicity Principle, that is the core of this paper, fall in the class of non-iterative imaging methods. Colton and Kirsch introduced the first non-iterative approach named Linear Sampling Method (LSM) \cite{art:Co96} followed by the Factorization Method (FM) proposed by Kirsch \cite{art:Ki98}; Ikehata proposed the Enclosure Method (EM) \cite{ikehata1999draw,Ikehata_2000,art:Ikh02} and Devaney applied MUSIC (MUltiple SIgnal Classification), a well known algorithm in signal processing, as imaging method \cite{Devaney2000}. 

In this framework of real-time imaging methods, a key role is played by the Monotonicity Principle, that states a monotonic relationship between the pointwise value of the spatial distribution of the material property and a proper boundary operator \cite{art:Gi90, art:Ta02}. In the case of Electrical Resistance Tomography (ERT) in the presence of nonlinear materials, the material property is the electrical conductivity, whereas the boundary operator is the so-called average Dirichlet-to-Neumann (ADtN) operator \cite{MPpiecewise,art:Co21}, a suitable generalization of the classical Dirichlet-to-Neumann operator to nonlinear cases. Analogously, an ADtN can be introduced also in the field of MIT and ECT, as deeply discussed in \cite{art:MPNL2024}.

Monotonicity based imaging methods for linear materials find applications in a wide range of problems modeled by different PDEs, from static (elliptic PDEs) to wave propagation (hyperbolic) problem, including quasi-static (parabolic) cases. The Monotonicity Principle Method (MPM)  has first been proposed in \cite{art:Ta02} for ERT, a problem governed by an elliptic PDE, and developed for static problems such as Electrical Capacitance Tomography and Inductance Tomography, as well as Electrical Resistance Tomography \cite{art:calvano2012, Garde01092020, art:garde2022,gardevogelius2024, art:Ta03}. Then, it has been extended to quasi-static regimes governed by elliptic-parabolic PDEs \cite{art:tamburrino2006}, such as Eddy Current Tomography. In the latter case, MPM was proved for Eddy Current Tomography in the low-frequency (large skin-depth) limit \cite{art:Ta06}, in the high-frequency (small skin-depth) limit \cite{Tamburrino_2010} and in time domain (Pulsed Eddy Current Tomography) operations \cite{conf:tamburrino2017,suudpa2023monotonicity, art:Ta17,tamburrino2021monotonicity,art:su2015,art:ventre2015}.

Other extensions of the Monotonicity Principle can be found in~\cite{art:albicker2020,art:albicker2023,art:arens2023,art:daimon2020,art:griesmaier2018,art:Ha19, art:meftahi, AT_WAVE2015} for the Helmholtz equation, in~\cite{art:Ha21} for linear elasticity equations and in \cite{art:kar} for the quasilinear generalizations of the classical biharmonic operator. 

Finally, in~\cite{art:Co21} and~\cite{MPpiecewise}, the Monotonicity Principle has been introduced for nonlinear problems, under quite general assumptions on the material property. The related imaging method, together with realistic numerical examples, can be found in \cite{art:MPNL2024}.

The study of inverse problems involving nonlinear Maxwell's equation arose only in recent years. According to our awareness, there are very only few works on this research topic, as clearly stated in \cite{lam2020consistency}:
\lq\lq... the mathematical analysis for inverse problems governed by nonlinear Maxwell's equations is still in the early stages of development.\rq\rq .

With reference to Electrical Resistance Tomography, some results for the $p$-Laplacian, i.e. when the electrical conductivity is of the type
\begin{equation}
    \sigma(x,E)=\theta(x)E^{p-2},
\end{equation}
are available. Specifically, the inverse problem of retrieving $p-$Laplacian electrical conductivity from boundary measurements was first posed in~\cite{art:Sa12}, where the authors prove that the value of electrical conductivity on the boundary is uniquely determined by a nonlinear DtN operator. In~\cite{art:Bra14} the authors extend the uniqueness result to the first order derivative. Furthermore, an inversion algorithm was given in~\cite{art:Bra15}, where the authors studied the enclosure method for the $p$-Laplacian to reconstruct the convex hull of an inclusion. In~\cite{art:Bra17, art:Bra15, art:Salo16}, an ad-hoc version of the Monotonicity Principle for the $p$-Laplacian was derived, and in~\cite{art:Bra17} a Monotonicity Principle based reconstruction method for retrieving the complex hull of inclusions was proposed. For the sake of completeness, in~\cite{art:Hauer15} the properties of DtN operator, when $\theta(x)=1$, are discussed, while in~\cite{art:Car20} the authors proved that the Calder\'on problem admits a unique solution in the specific case of a nonlinearity given by a linear term plus a $p-$Laplacian term and gave a procedure for reconstructing the electrical conductivity.

The mathematical framework in which the present paper fits is more general than that recently developed in the theoretical contributions listed above. Indeed, the nonlinearity of the anomaly can belong to three different wide classes of constitutive relationships, where the nonlinear
material property, as function of $s=\abs{\nabla u}$, can be
either bounded from infinity and zero, or bounded from zero only, or bounded by
infinity only. In each of these classes, the nonlinearity is not specific but, rather, it may be very general (see \Cref{sec:ass}), as long as the $s \mapsto \gamma_{NL}(x,s)s$ is monotonic for almost every $x \in \Omega$. 

This contribution integrates with the findings of previous works \cite{art:Co21,MPpiecewise,art:MPNL2024}, on the topic of the inverse obstacle problem for nonlinear inclusions in a linear background, via the MP. 

The main contribution of this work consists in proving that (i) the imaging method is stable and robust with respect to the noise, (ii) the reconstruction approaches monotonically to a well-defined limit, as the noise level approaches to zero, and that (iii) the limit contains the unknown set and is contained in the outer boundary of the unknown set.

Results (i) and (ii) come directly from the Monotonicity Principle, while results (iii) requires to prove the so-called Converse of the Monotonicity Principle. The Converse of the Monotonicity Principle, here proved for nonlinear inclusions for the first time, is a theoretical results of fundamental relevance to evaluate the ideal (noise-free) performances of the imaging method. It is a non-trivial results whose proof poses relevant challenges. In few words, the Converse of MP states that an anomaly can be perfectly reconstructed, from noise-free data, apart from the cavities of the anomaly that are not connected by arch to the boundary where the data are collected.

The results are provided in a quite general setting, and proved for three wide classes of constitutive relationships where nonlinearity can be bounded from infinity and zero, or bounded from zero only, or bounded by infinity only. These classes of constitutive relationships cover the wide majority of cases encountered in applications.

The paper is organized as follows. In \Cref{sec:uiaer}, an overview on the key results is given. In
Section \ref{sec:def}, the mathematical foundation of the problem is described. In Section~\ref{sec:main}, the Converse of the MP is proven. This section is divided into four subsections in which the Converse of the Monotonicity Principle is proved for different classes of nonlinearities. In Section \ref{sec:stab}, the intrinsic stability and the robustness of the reconstructions with respect to measurement noise is proved. In Section \ref{sec:conclusions}, the conclusions are drawn.

\section{Overview on key results}
\label{sec:uiaer}
In this Section, the main results achieved in this contribution are briefly described. In \Cref{sec:mpm} it is summarized the Monotonicity Principle, cast for the present setting. In \Cref{sec:reconstruction} the reconstruction method and its features are described. In \Cref{sub_converse} the Converse of the MP and its impact are discussed. 

The region under tomographic inspection is termed $\Omega$. Let $\Omega\subseteq \mathbb{R}^N$, $N\geq 2$, be an open bounded domain with a Lipschitz continuous boundary. Similarly, the region occupied by the anomaly is termed $A$ and it is assumed that $A$ is well contained in $\Omega$, is an open bounded set with a Lipschitz boundary and $\partial A$ is made by a finite number of connected components, i.e. $A \in \mathcal{S}(\Omega)$, where
\begin{equation}\label{SOmega}
\begin{split}
\mathcal{S} (\Omega):= & \left\{ V \subseteq \Omega: V \text{ is an open bounded set with a Lipschitz boundary} \right. \\ 
&\  \left. \text{and } \partial V \text{ is made by a finite number of connected components}  \right\}.
\end{split}
\end{equation}

\subsection{The Monotonicity Principle}\label{sec:mpm}
MP consists in a monotone
relation (see \cite{MPpiecewise,art:Co21,art:MPNL2024} for details) connecting the region occupied by the nonlinear material to the measured boundary operator: 
\begin{equation}
\label{eqn:mono2}
    T\subseteq A \Longrightarrow  \overline{\Lambda}_{T}\leqslant\overline{\Lambda}_{A}.
\end{equation}
In \eqref{eqn:mono2}, it has been assumed that $\gamma_{NL}$ is greater than $\gamma_{BG}$, $A$, $T \in \mathcal{S} (\Omega)$, $\overline{\Lambda}_A$ and $\overline{\Lambda}_T$ are the so-called  \emph{average} DtN operators (see \cite{art:Co21}, ~\cite{MPpiecewise} and  the following \Cref{sec:dtns}). Inequality $\overline{\Lambda}_{T}\leqslant\overline{\Lambda}_{A}$ is intended as
\begin{equation*}
\overline{\Lambda}_{T}\leqslant\overline{\Lambda}_{A} \Longleftrightarrow \left\langle \overline{\Lambda}_{T}(f)-\overline{\Lambda}_{A}(f),f\right\rangle \leq 0 \quad \forall f\in X_\diamond,
\end{equation*}
where $X_\diamond$ is a proper functional space, defined in \Cref{MathModel}. It is worth noting that the average DtNs are nonlinear operators in the presence of nonlinear materials.

\subsection{The reconstruction method and its features}\label{sec:reconstruction}
The imaging method is based on the following equivalent form of \eqref{eqn:mono2}:
\begin{equation}\label{eqn:mono3}
    \overline{\Lambda}_{T}	\nleqslant \overline{\Lambda}_{A}\Longrightarrow T\nsubseteq A.
\end{equation}

Relation \eqref{eqn:mono3} allows to infer when a test domain $T$ is not included in the unknown anomaly $A$, starting from boundary data. Specifically, it elevates inequality $\overline{\Lambda}_{T}	\nleqslant \overline{\Lambda}_{A}$ as a sufficient condition to infer when $T \nsubseteq A$.

From the monotonicity test of~\eqref{eqn:mono3}, a reconstruction method can be obtained by repeating the test for a set of test domains $T$, covering the region of interest, i.e. the estimate $A^{\dag}$ of the inclusion $A$ is$^1$ \footnote{$^1$Symbol $\dag$ is borrowed from the classical theory of ill-posed problem, and it refers to the solution without any type of regularization.}
\begin{align}  \label{eqn:recon}
  A^{\dag}&=\bigcup \left\{T \in \mathcal S(\Omega)\,|\,\overline{\Lambda}_{A}-\overline{\Lambda}_{T}\geqslant 0\right\}. 
\end{align}

It results that $A^{\dag}$ is an upper bound to $A$, i.e.
\begin{equation}\label{eqn:AAdag}
    A \subseteq A^{\dag},
\end{equation}
because $A \in \mathcal S(\Omega)$.

The reconstruction method of \eqref{eqn:recon} is valid in the absence of noise, but in any practical application, the noise corrupts the data giving a void reconstruction, i.e. $A^{\dag} = \emptyset$ (see \Cref{adagnullo}). As discussed in detail in \cite{Garde2017,garde2019theregularizedmm} for linear coefficients and an additive noise model, MP can be naturally regularized and stabilized to treat noisy data and modelling errors.

In this paper, it is showed that noise can be treated also in the nonlinear case. 
Specifically, in line with \cite{art:MPNL2024}, the noisy data are assumed to be modeled as
\begin{equation}\label{eqn:noisemodel}
\mathcal P^{\bm{\eta}}_A(f):=
\mathcal P_A(f) (1+\eta_1\xi_1)+\eta_2\xi_2L,
\end{equation}
where $\mathcal P^{\bm{\eta}}_A(f)$ is the noisy version of the so-called power product:
\begin{equation}\label{eqn:noisemes}
\mathcal P_A(f):=\langle \overline{\Lambda}_A(f),f\rangle,    
\end{equation}
$\eta_1$, $\eta_2$ are two positive constants, $\bm{\eta}=(\eta_1,\eta_2)$, and $\xi_1$, $\xi_2$ are two random variables uniformly distributed in $[-1,1]$. This noise model is common in current digital instruments and equipment.
In \eqref{eqn:noisemodel}, two distinct terms are present: one controlled by $\eta_1$, which is proportional to the measured value, and one controlled by $\eta_2$, which is proportional to the measurement range $L$ of the specific instrument. Both parameters are usually provided by the manufacturer of the instrument. 

In the presence of noise, the reconstruction rule changes in (see \cite{art:MPNL2024} for details)
\begin{align}
\label{eqn:recrulenoise}
{A}^{\bm{\eta}}_{\bm{\eta}} & = \bigcup \left\{ T \in \mathcal S(\Omega) \ :  \frac{\mathcal P^{\bm{\eta}}_A(f)+\eta_2 L}{1-\eta_1}-\mathcal P_T(f)
\geq 0 \ \ \forall f\in X_\diamond\right\},
\end{align}
It is worth noting that ${A}^{\bm{\eta}}_{\bm{\eta}}$ is a random set, because of the presence of the noise.

In \Cref{sec:stab} is proved that
\begin{equation}
\label{eqn:01}
    \Aps \subseteq {A}^{\bm{\eta}}_{\bm{\eta}}\subseteq \mathbb A_{\bm{\eta}}^{\bm{\eta}},
\end{equation}
that $\{ A_{\bm{\eta}_k}^{\bm{\eta}_k} \}_{k \in \mathbb{N}}$ is a monotonic and convergent sequence for $\eta _{1,k}\downarrow 0^{+}$ and $\eta _{2,k}\downarrow 0^{+}$, i.e. 
\begin{align}
\label{eqn:02}
   \mathbb A_{\bm{\eta}_{k+1}}^{\bm{\eta}_{k+1}} & \subseteq \mathbb A_{\bm{\eta}_k}^{\bm{\eta}_k} \\
\label{eqn:03}
   \overline{\lim_{k\to+\infty} \mathbb A_{\bm{\eta}_k}^{\bm{\eta}_k}} & = \overline{\Aps} \\
\label{eqn:conv}
   \overline{\lim_{k\to+\infty}  A^{{\bm{\eta}_k}}_{\bm{\eta}_k}} & = \overline{\Aps},
\end{align}
where $\mathbb A_{\bm{\eta }}^{\bm{\eta }}$ is a properly defined deterministic set. Specifically, the set $\mathbb A_{\bm{\eta }}^{\bm{\eta }}$ is deterministic and it is equal to $A_{\bm{\eta }}^{\bm{\eta }}$, when all realizations of the random variables $\xi _{1}$ and $\xi _{2}$ are equal to $1$
(for the precise definition see \Cref{sub_sec_regularization}).

Equations \eqref{eqn:01} and \eqref{eqn:03} prove that the imaging method is stable. Indeed, the reconstruction ${A}_{\bm{\eta}}^{\bm{\eta}}$ is constrained within a lower bound $\Aps$ and an upper bound $A_{\bm{\eta}}^{\bm{\eta}}$, regardless the noise realization. Moreover, the upper bound tends monotonically (see \eqref{eqn:02}) to the lower bound, as the noise level $\bm{\eta}$ approaches zero (see \eqref{eqn:03}), thus implying that the reconstruction ${A}_{\bm{\eta}}^{\bm{\eta}}$ approaches $\Aps$, the theoretical limit in ideal conditions.
The MP based imaging method is, therefore, \emph{unconditionally stable} with respect to the noise, even in the presence of nonlinear materials.


\subsection{The Converse of MP}\label{sub_converse}
In this paper, a key contribution is the proof of the Converse of the Monotonicity Principle in the presence of nonlinear materials (see \Cref{sec:main}). Specifically, it is proved that $\overline{\Lambda}_{T}	\nleqslant \overline{\Lambda}_{A}$ is a necessary condition for $T \nsubseteq A^{\ast}$, i.e.
\begin{equation}\label{eqn:conv0}
    T\nsubseteq A^{\ast} \Longrightarrow \overline{\Lambda}_A \nleqslant \overline{\Lambda}_T,
\end{equation}
where $A^{\ast}$ is the so-called outer support of the anomaly $A$, introduced in \cite{art:Ha13} and discussed in \Cref{sec:def}. Intuitively, set $A^{\ast}$ coincides with $A$, plus all internal cavities (of $A$) that are not touching $\partial\Omega$ (see Section \ref{sec:def} and Figure \ref{fig:outerM} for details).

In the framework of linear inclusions embedded in a linear background, the converse has been proven in~\cite{art:Ha13}. This work proves the above mentioned results in the more general and complex case of nonlinear inclusions, extending the range of applications for electrical and magnetic tomography and real-time inversion methods. The class of nonlinearities that can be treated in this framework is quite wide: the essential requirements are that (i) $s \mapsto \gamma(x,s)s$ has to be a monotonic function in $s$, and that (ii) $\gamma(x,s)$ is bounded by a power function. Proper details are provided in Section \ref{sec:ass}. 
The class of nonlinearities that can be treated within the present framework includes piecewise and rational functions, too.

The Converse of the Monotonicity Principle has a paramount role from the applications perspective, other than from the theoretical one. Indeed, thanks to the theoretical result of the Converse of the MP, it is possible to prove that
\begin{align}
\label{eqn:cmp1}
\Aps & \subseteq A^{\ast}.
\end{align}
and, therefore, \eqref{eqn:cmp1} combined with \eqref{eqn:AAdag}, gives the theoretical limit of the Monotonicity Principle Method in ideal conditions:
\begin{equation}\label{eqn:inclusions}
A \subseteq \Aps \subseteq A^{\ast}. 
\end{equation}
This means that the Monotonicity Principle Method reconstructs $A$, plus some of its internal cavities that are not connected to the boundary $\partial \Omega$ (see \Cref{sub_sec_support} for the concept of outer support). In any case, the reconstruction never exceeds the outer support $A^{\ast}$, setting the theoretical limit of the capabilities of the method.

Equation \eqref{eqn:conv}, combined with \eqref{eqn:inclusions}, proves that, even in the presence of noise, the imaging rule of \eqref{eqn:recrulenoise} reconstructs a set bounded by $A$ and $\Aout$, as aforementioned.

\section{Mathematical Framework}\label{sec:def}
This contribution is focused on a nonlinear inverse obstacle problem of great interest in applications, consisting in retrieving nonlinear anomalies embedded in a known linear background. The considered material property is, therefore, given by
\begin{equation}\label{eqn:sigmaa}
  \gamma_A(x,s)=\begin{cases}
   \gamma_{BG}(x) & \text {in $\Omega\setminus A$}, \\
   \gamma_{NL}(x,s) & \text {in $A$}, \\
  \end{cases}
\end{equation}
where $A\subset \Omega$ is the unknown region occupied by the nonlinear anomalies.

\subsection{Assumptions}\label{sec:ass}

Before giving the assumptions on the material property \eqref{eqn:sigmaa}, it is convenient to recall the definition of the Carath\'eodory function.
\begin{definition}
\label{def:asmp}
$\gamma:\Omega\times [0,+\infty)\to\mathbb R$ is a Carath\'eodory function in $\Omega$ iff:
\begin{itemize}
\item $x\in\Omega\mapsto \gamma(x,s)$ is measurable for every $s\in[0,+\infty)$,
\item $s\in [0,+\infty)\mapsto \gamma(x, s)$ is continuous for almost every $x\in\Omega$.
\end{itemize}
\end{definition}

It is assumed that $\gamma_{BG}\in L^{\infty}_+(\Omega)=\{u\in L^{\infty}(\Omega) : u\geq c_0>0\text{ a.e. in $\Omega$}\}$ and $\gamma_{NL}: A\times[0,+\infty)\to\mathbb{R}$ satisfies the following assumptions (see~\cite{MPpiecewise}):
\begin{enumerate}
\item[{\bf (A1)}] $\gamma_{NL}$ is a Carath\'eodory function in $\Omega$;
\item[{\bf (A2)}] $s\in [0,+\infty) \mapsto \gamma_{NL}(x,s)s$ is strictly increasing for a.e. $x\in A$.
\item[{\bf (B1)}] 
There exist two positive constants $\underline{\gamma}\le\overline{\gamma}$
such that
\begin{equation*}
\underline{\gamma}\leq\gamma_{NL}(x,s)\leq \overline{\gamma}\quad \text{ for a.e. } x\in A\ \text{and}\ \forall s>0.
\end{equation*}
\item[{\bf (B2)}] For fixed $1<q<+\infty$, there exist three positive constants $\underline{\gamma}\le\overline{\gamma}$ and $s_0$ 
such that
\begin{equation*}  
\underline{\gamma} \leq\gamma_{NL}(x, s)\leq 
\begin{cases}
    \overline{\gamma}\left[1+\left( \frac{s}{s_0} \right)^{q-2}\right] & \text{if}\ q\ge 2,\\
    \overline{\gamma}\left( \frac{s}{s_0} \right)^{q-2} &  \text{if}\ 1<q< 2,
\end{cases}
\end{equation*}
$\text{for a.e.}\ x\in {\overline A}\ \text{and}\ \forall s>0$.
\item[{\bf (B3)}] For fixed $2\leq q<+\infty$, there exist three positive constants $\underline{\gamma}\le\overline{\gamma}$ and $s_0$ 
such that
\begin{equation*}   
\underline{\gamma} \left(\frac{s}{s_0}\right)^{q-2}\leq \gamma_{NL}(x, E)\leq \overline{\gamma}
\end{equation*}
$\text{for a.e.}\ x\in {\overline A}\ \text{and}\ \forall s>0$.
\item[{\bf (C1)}] 
There exists a constant $\kappa>0$ such that
\begin{equation*}
(\gamma_{NL}(x,s_2){\bf s}_2-\gamma_{NL}(x,s_1){\bf s}_1)\cdot( {\bf s}_2-{\bf s}_1)\geq\kappa|{\bf s}_2-{\bf s}_1|^2
       \end{equation*}
       \ $\text{for a.e.}\ x\in A$ and for any ${\bf s}_1,{\bf s}_2\in\mathbb{R}^n$.
\item[{\bf (C2)}] For fixed $1<q<+\infty$, there exists a constant $\kappa>0$ such that
\begin{equation*}\begin{split}(\gamma_{NL}(x,s_2)&{\bf s}_2-\gamma_{NL}(x,s_1){\bf s}_1)\cdot( {\bf s}_2-{\bf s}_1)\\        &\geq        \begin{cases}       \kappa|{\bf s}_2-{\bf s}_1|^q\ &\text{if} \ q\geq 2\\ \kappa(1+ |{\bf s}_2|^2+|{\bf s}_1|^2)^\frac{q-2}2|{\bf s}_2-{\bf s}_1|^2\ &\text{if}\ 1<q<2\end{cases}\\\end{split}\end{equation*} \ $\text{for a.e.}\ x\in A$ and for any ${\bf s}_1,{\bf s}_2\in\mathbb{R}^n$.
\item[{\bf (C3)}] For fixed $2\leq q<+\infty$, there exists a constant $\kappa>0$ such that
\begin{equation*}(\gamma_{NL}(x,s_2){\bf s}_2-\gamma_{NL}(x,s_1){\bf s}_1)\cdot( {\bf s}_2-{\bf s}_1)\geq           \kappa|{\bf s}_2-{\bf s}_1|^q.
\end{equation*} \ $\text{for a.e.}\ x\in A$ and for any ${\bf s}_1,{\bf s}_2\in\mathbb{R}^n$.
\end{enumerate}
The above hypothesis take into account bounded and possibly unbounded or vanishing nonlinear material properties. Assumptions (A1) and (A2) hold in each case. Assumptions (B1)(C1), (B2)(C2) and (B3)(C3) are alternative to each other.

\begin{remark}
The assumptions are largely general, in the sense that they can accommodate wide classes of nonlinearities. Examples are polynomial nonlinearities (see \cite{art:Co21}) or sigmoids, for instance.

From a general perspective, assumptions (Ax) are required to obtain the \emph{existence and uniqueness} of the solution of the forward problem, i.e. of the scalar potential $u$. Assumptions (Bx) provide upper and/or lower bounds to the material property. Assumption (Cx) corresponds to the strict monotonicity of the vector-valued function $\mathbf{s} \mapsto \gamma_{NL}(\mathbf{s})$ (see \cite{Deimling1985} for the concept of monotonicity of operators and vector-valued function).
Summing up, other than quite standard requirement on function $\gamma_{NL}$, it is required (i) the control of $\gamma_{NL}$ from below and/or above (assumptions Bx) and (ii) the strict monotonicity of the vector-valued mapping (assumptions Cx).
\end{remark}

\subsection{The mathematical model}
 \label{MathModel}

In the mathematical model \eqref{P1} the prescribed Dirichlet data $f$ is an element of \[
X_{\diamond}=\left\{g\in H^{1/2}(\partial\Omega): \int_{\partial\Omega}g\, dS=0\right\}
\]
and $\partial_\nu$ denotes the outer normal derivative on $\partial\Omega$. 
It is worth noting that 
$u_A$ belongs to 
$H^1(\Omega)$. 

Problem~\eqref{P1} is understood in the weak form, i.e.
\begin{equation}\label{weakform}
    \int_{\Omega\setminus A}\gamma_{BG}(x)\nabla u_{ A}(x)\cdot \nabla\psi (x)\,dx+\int_{A}\gamma_{NL}(x,\nabla u_A(x))\nabla u_A(x)\cdot \nabla\psi(x)\,dx=0 
\end{equation}
for any $\psi \in C^{\infty}_0(\Omega)$.
The unique weak solution $u_A$ of the problem~\eqref{weakform} is variationally characterized as
\begin{equation}
u_A=\argmin{\{\mathbb{E}_A(u) : u\in H^{1}(\Omega),\,  u\rvert_{\partial\Omega}=f\in X_{\diamond}\}}.
\label{eqn:varprob}
\end{equation}
The functional $\mathbb{E}_A$ to be minimized is the Dirichlet Energy
\begin{equation}
\label{eqn:Ea}
\mathbb{E}_A(u):=\int_{\Omega}\int_0^{\abs{\nabla u(x)}}\gamma_A(x,\eta)\eta\,d\eta\,dx.
\end{equation}
By recalling \eqref{eqn:sigmaa}, it results that
\begin{equation*}
\mathbb{E}_A(u)=\int_{\Omega\setminus A} Q_{\gamma_{BG}}\left(x,\abs{\nabla u(x)}\right)\,dx+\int_{A} Q_{\gamma_{NL}}\left(x,\abs{\nabla u(x)}\right)\,dx
\end{equation*}
where
\begin{align*}
&Q_{\gamma_{NL}}(x,s)=\int_{0}^{s} \gamma_{NL}(x,\eta)\eta\,d\eta\quad\text{for a.e. $x\in A$ and $\forall s\geq 0$},\\
&Q_{\gamma_{BG}}(x,s)=\frac{1}{2}\gamma_{BG}(x)s^2 \quad \text{for a.e. $x\in\Omega\setminus A$ and $\forall s\geq 0$}.
\end{align*}
Existence and uniqueness of the solution of \eqref{weakform} is discussed in \cite{MPpiecewise}.

\subsection{The DtN operators}
\label{sec:dtns}
The Dirichlet-to-Neumann (DtN) operator maps the Dirichlet data into the corresponding Neumann data:
\begin{equation*}
\Lambda_A  :f\in X_\diamond\mapsto \gamma_A(x, |\nabla u_A|)\ 
\partial_nu_A|_{\partial\Omega} 
\in X_\diamond',
\end{equation*}
where $X_\diamond'$ is the dual space of $X_\diamond$ and $u_A$ is the solution of \eqref{P1}. 
From a physical point of view, the DtN operator maps the imposed boundary data to the quantity measured on the boundary $\partial\Omega$. For instance, in ERT the DtN maps the imposed boundary electric scalar potential to the normal component of the electrical current density entering $\partial \Omega$ (see \cite{art:MPNL2024} for a detailed discussion in the case of MIT and ECT).

In weak form, the DtN operator is
\begin{equation}
\label{w-DtN}
\langle \Lambda_A  \left( f\right) ,\psi\rangle
=\int_{\partial \Omega }\psi (x) \gamma_A\left( x, \left\vert \nabla u_A(x)\right\vert\right)  \partial_n u_A(x)\,\text{d}S\quad\forall \psi\in X_\diamond.
\end{equation}

Furthermore, by testing the DtN operator \eqref{w-DtN} with the solution $u$ of \eqref{P1} and using a divergence Theorem, it results 
\begin{equation}\label{w-DtN-f}
\langle \Lambda_A  \left( f\right) ,f\rangle
=\int_{\Omega } \gamma_A( x ,\nabla u_A(x)){ |\nabla u_A}(x)|^2\ \text{d}x.
\end{equation}

The Average DtN (ADtN) is defined as (see \cite{art:Co21,MPpiecewise})
\begin{equation}
\label{P4}
\overline{\Lambda}_A: f\in X_{\diamond} \to \int_0^1\Lambda_A(\alpha f)\,d\alpha\in X'_\diamond,\qquad \overline{\Lambda}_T: f\in X_{\diamond} \to \int_0^1\Lambda_T(\alpha f)\,d\alpha\in X'_\diamond,
\end{equation}
where
\begin{equation}
\label{P3}
    \Lambda_A:f\in X_{\diamond}\to \gamma_{A}\partial_n u_A|_{\partial\Omega}\in X'_\diamond, \qquad     \Lambda_T:f\in X_{\diamond}\to \gamma_{T}\partial_n u_T|_{\partial\Omega} \in X'_\diamond
\end{equation}
are the classical DtN operators related to anomalies occupying regions $A$ and $T$, respectively.

In equations \eqref{eqn:mono2}, \eqref{P4} and \eqref{P3} $\overline{\Lambda}_{A}$, $\Lambda_A$ and $u_A$ refer to $\gamma_A$,
 whereas $\overline{\Lambda}_{T}$, $\Lambda_T$ and $u_T$ refer to $\gamma_T$ given by 
\begin{displaymath}
    \gamma_T(x,s)=\begin{cases}
        \gamma_{BG}(x) & \text{in $\Omega\setminus T$}\\
                \gamma_{NL}(x,s) & \text{in $T$}.
    \end{cases}
\end{displaymath}

The functional $f \mapsto \langle \overline{\Lambda}  \left( f\right) ,f\rangle$ is termed power product, in line with \cite{art:Co21}. From the mathematical standpoint, the power product gives the so-called Dirichlet Energy (see~\cite{art:Co21, MPpiecewise}):
\begin{equation}
    \label{ADtN=E}
        \left\langle \overline{\Lambda}_A(f), f\right\rangle=\mathbb{E}_A(u_A).
\end{equation}
On the other hand, from the physical standpoint, the power product corresponds to the  ohmic power dissipated in $\Omega$ for ERT, whereas it gives the electrostatic co-energy and the magnetostatic co-energy, for ECT and MIT, respectively.

\subsection{The outer support of a set}\label{sub_sec_support}
For the convenience of the reader, the concept of outer support~\cite{art:Ha13} of a set $A\subseteq \Omega$ is reminded. The following definition is equivalent to that of ~\cite{art:Ha13}, but is simpler and more intuitive.

\begin{definition}\label{def:outer}
The outer support of a set $A \subseteq \Omega$, denoted as $\outi_{\partial\Omega}{A}$, is the complement in $\overline{\Omega}$, of the union of those relatively open set $U$ contained in $\Omega \setminus\overline A$ and connected to $\partial \Omega$, i.e. those sets $U$ that are connected and satisfying $\partial U\cap\partial\Omega \neq \emptyset$. 
\end{definition}
In the following, for the sake of simplicity, the outer support is denoted with a $*$ superscript, i.e.
\begin{equation}\label{out_supp_c}
    A^* \equiv \outi_{\partial\Omega}{A}.
\end{equation}

\begin{remark}
It is worth noting that all the boundary points of $A^*$ are \emph{connected} to $\partial \Omega$. Moreover, when $A$ does not contain cavities that are not connected by arch to $\partial \Omega$, it results that
\begin{equation}
    A^*=A,
\end{equation}
refer also to \Cref{fig:outerM}.
\end{remark}

\begin{figure}[htp]
    \centering
    \subfloat[][]
    {\includegraphics[width=.3\textwidth]{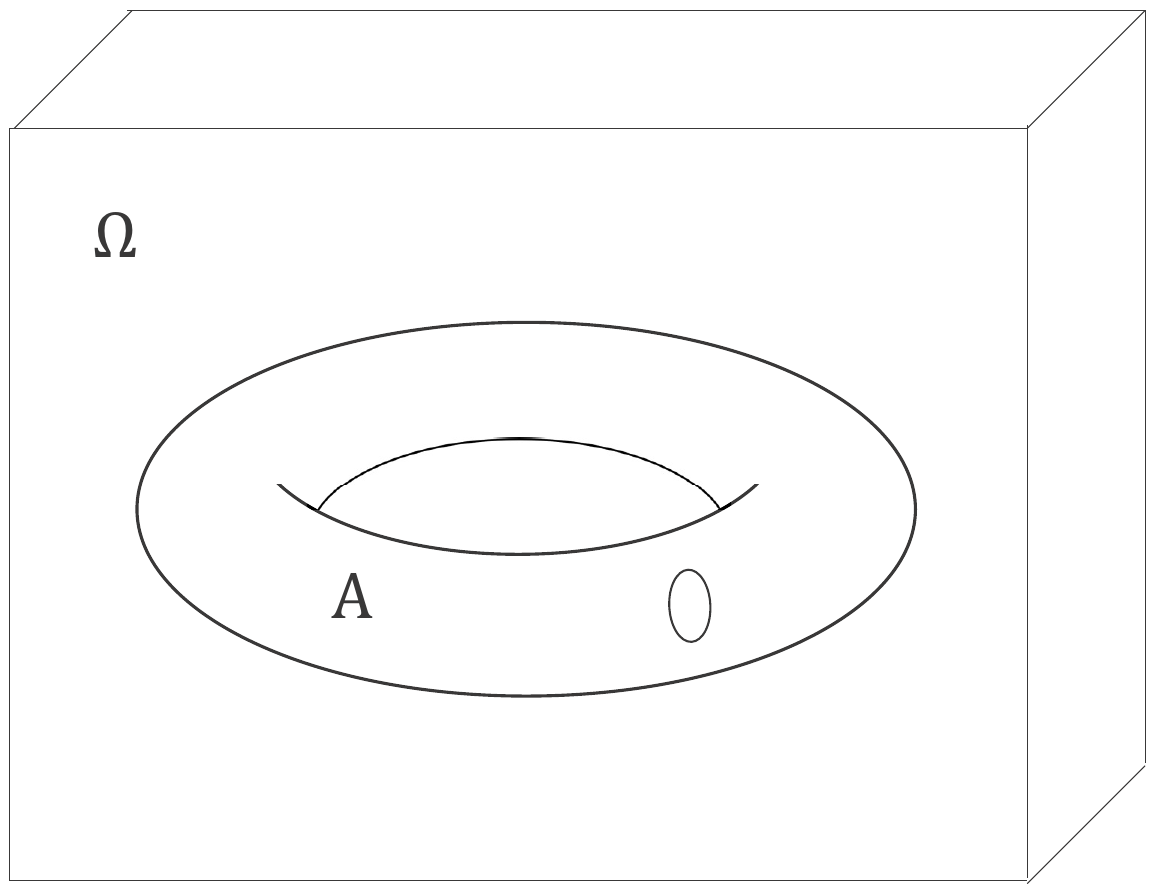}} \quad
    \subfloat[][]
    {\includegraphics[width=.3\textwidth]{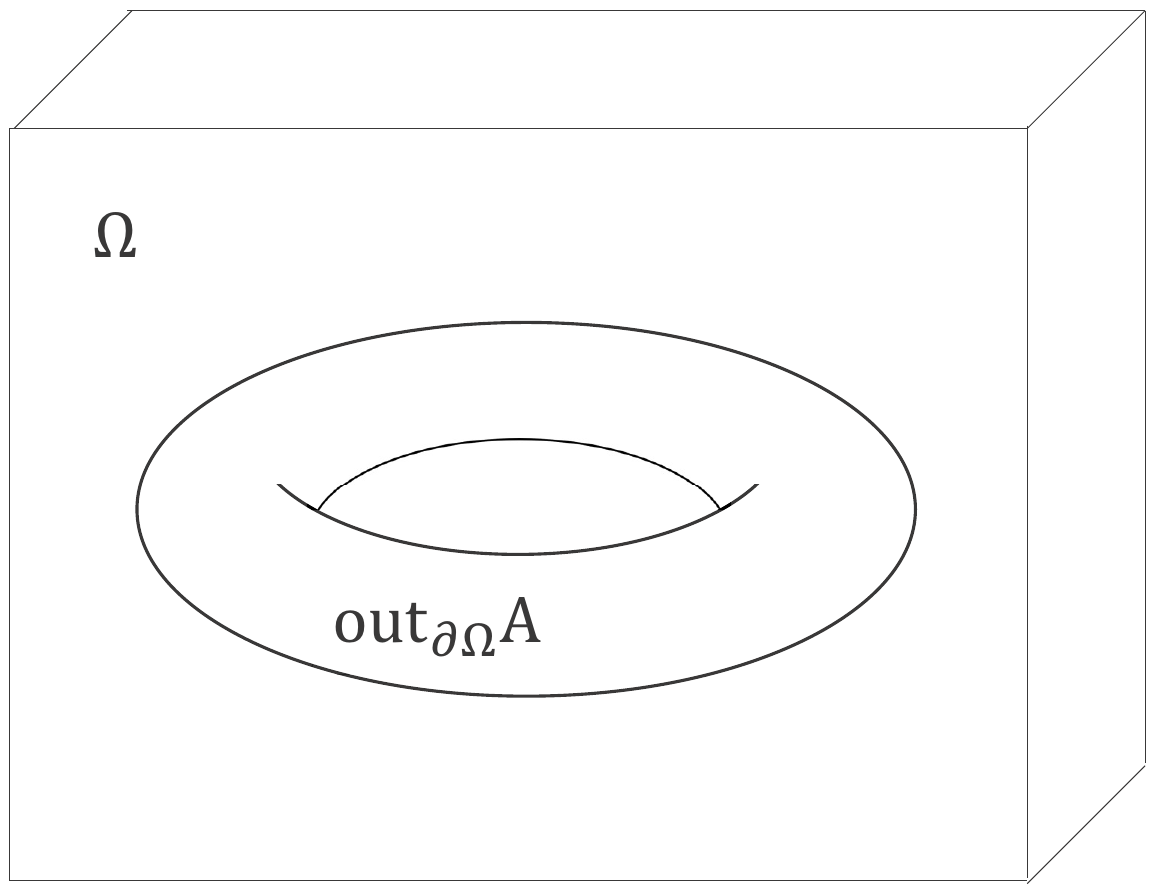}} \quad
    \subfloat[][]
    {\includegraphics[width=.3\textwidth]{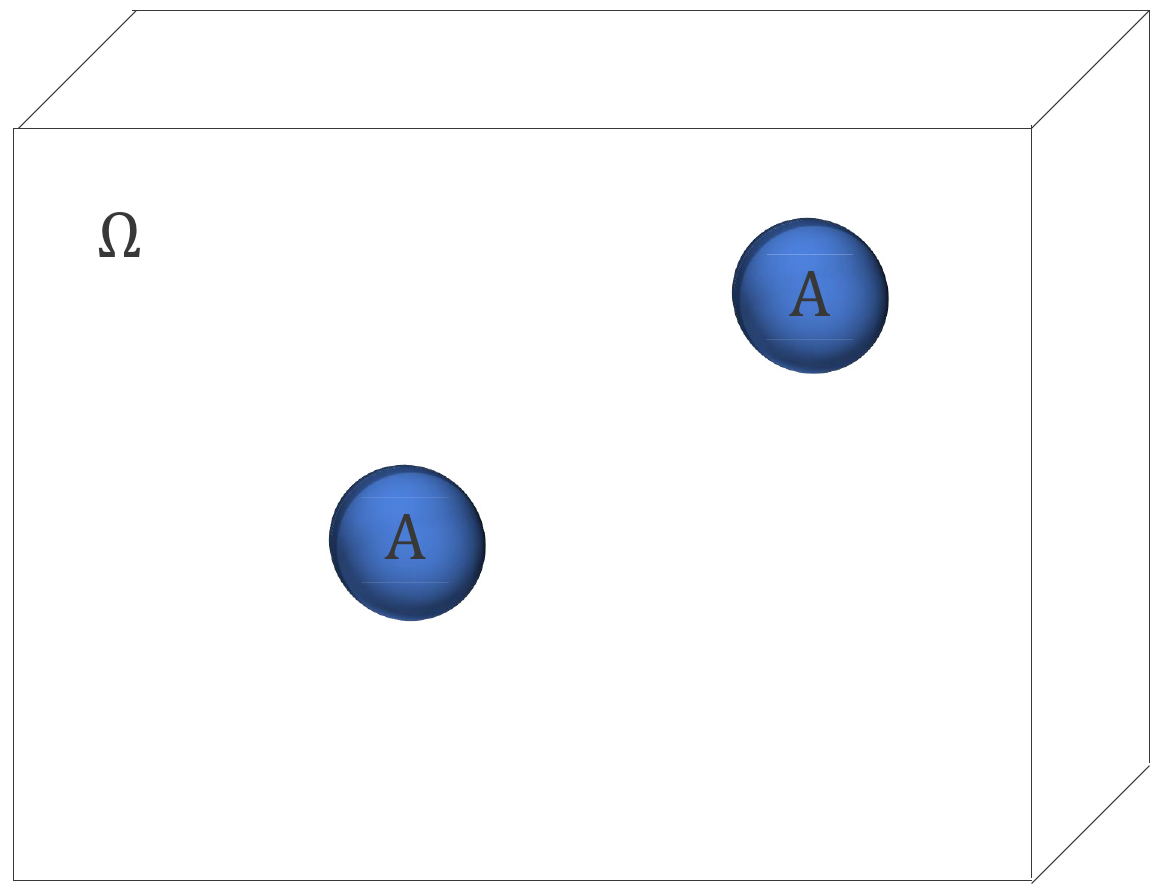}} 
    \caption{Left: Anomaly $A$ represented by a torus with a void inside. Center: outer support of $A$. Right: a set $A$, made by several connected components, that does not have any cavity and coincides with $A^*$.}
    \label{fig:outerM}
\end{figure}

\section{Converse of the Monotonicity Priciple}
\label{sec:main}
Before stating the main result, the concept of localized potentials is extended from Neumann data \cite{art:Ge08} to Dirichlet data. Localized potentials have been exploited to prove the Converse of the Monotonicity Principle in the linear case~\cite{art:Ha13}. 

\begin{proposition}
\label{lem:bridge}
Let $S_1,S_2\subset\Omega$ be two open sets such that $\overline{S_1}\cap\overline{S_2}=\emptyset$ and $\Omega\setminus(\overline{S_1}\cup \overline{S_2})$ is connected. Let the linear material property $\gamma\in L^{\infty}_+(\Omega)$ be piecewise analytic. Then there exists a sequence $\{f_n\}_{n\in\mathbb{N}}\subset X_{\diamond}$ of boundary potentials such that the family of solutions $\{u_n\}_{n\in\mathbb{N}}$ of the following problem
\begin{equation}
\label{eqn:lemfn}
\begin{cases}
    \nabla\cdot(\gamma(x)\nabla u_n(x))=0 & \text{in $\Omega$} \\
    u_n(x)=f_n(x) & \text{on $\partial\Omega$}
\end{cases}
\end{equation}
fulfill
\begin{equation}\label{eqn:lemlim}
   \lim_{n\to+\infty}\int_{S_1} \gamma(x)\abs{\nabla u_n(x)}^2\,dx=0\quad\text{and}\quad\lim_{n\to+\infty}\int_{S_2} \gamma(x)\abs{\nabla u_n(x)}^2\,dx=+\infty.
\end{equation}
\end{proposition}
\begin{remark}
    It is worth noting that it is not mandatory for set $S_2$ to coincide with its outer support, i.e. $S_2$ is allowed to have cavities.
\end{remark}

\begin{proof}[Proof of Proposition~\ref{lem:bridge}]
The proof is based on the uniqueness of the solution for the Dirichlet problem and the existence of localized potentials for Neumann data~\cite[Theorem~3.6]{art:Ha13}.

Let $S_1,S_2\subset\Omega$ be two open sets such that $\overline{S_1}\cap\overline{S_2}=\emptyset$ and $\Omega\setminus(\overline{S_1}\cup \overline{S_2})$ is connected. Following~\cite[Theorem~3.6]{art:Ha13}, there exists a sequence $\{g_n\}_{n\in\mathbb{N}}\subset X_{\diamond}$ such that the solutions $v_n$ of the following problem
\begin{equation}
\label{eqn:pp4}
\begin{cases}
\nabla\cdot(\gamma(x) \nabla v_n(x))=0 & \text{in $\Omega$}\\
\gamma\partial_\nu v_n(x)=g_n(x) & \text{on $\partial\Omega$}\\
\int_{\partial\Omega}v_n(x)\,dx=0
\end{cases}
\end{equation}
fulfill
\begin{equation} \label{eqn:lpN}   \lim_{n\to+\infty}\int_{S_1} \gamma(x)\abs{\nabla v_n(x)}^2\,dx=0\quad\text{and}\quad\lim_{n\to+\infty}\int_{S_2} \gamma(x)\abs{\nabla v_n(x)}^2\,dx=+\infty.
\end{equation}

The Dirichlet data $f_n \in X_{\diamond}$ gives $u_n=v_n$ when plugged in problem \eqref{eqn:lemfn} and, therefore, \eqref{eqn:lpN} implies \eqref{eqn:lemlim}.
\end{proof}

In literature, some versions of localized potentials are present, with slightly different assumptions on $S_1$ and $S_2$. In particular, in~\cite{art:Ge08}, it is required that $\overline{S_1}\cap\overline{S_2}=\emptyset$ and $\Omega\setminus(\overline{S_1}\cup \overline{S_2})$ is connected; in~\cite{art:Ha13} it is introduced the notion of outer support for a set and the hypotheses become $\interior{S_2}\not\subseteq S_1^*$. In~\cite{art:Ga20}, the localized potentials are formulated as in the following: let $U\subseteq\overline{\Omega}$ a relatively open set that intersect the boundary with a connected complement, let $S_2\subset U$ and let $\sigma$ a linear piecewise analytic electrical conductivity, which stands for the linear material property $\gamma$, then there exists a sequence $\{g_n\}_{n\in\mathbb{N}}\subset X_{\diamond}$ such that the corresponding solutions $\{u_n\}_{n\in\mathbb{N}}$ of the problem
\begin{displaymath}
\begin{cases}
\nabla\cdot\left(\gamma(x)\nabla u_n(x)\right)=0\,\text{in $\Omega$}\\
\gamma(x)\partial_n u_n(x)=g_n(x)\,\text{on $\partial\Omega$}     
   \end{cases}
\end{displaymath}
fulfill
\begin{equation}\label{variante_loc}
\lim_{n\to+\infty}\int_{\Omega\setminus U}\abs{\nabla u_n(x)}^2\,dx=0    \quad \lim_{n\to+\infty}\int_{S_2}\abs{\nabla u_n(x)}^2\,dx=+\infty.
\end{equation}
The same arguments of the proof of \Cref{lem:bridge} can be applied indifferently to all these different formulations. In the following, some of these different formulations of localized potentials are used and, in doing that, \Cref{lem:bridge} is applied to the particular formulation of interest.

For the sake of clarity, the remaining of the section is divided into four subsections in which it is shown the Converse of the Monotonicity Principle, for different type of material properties for the nonlinear phase. Specifically, in Sections \ref{ano_1} and \ref{ano_2}, bounded nonlinear material properties are considered, under assumptions (B1)-(C1). In \Cref{ano_infty}, anomalies with possibly unbounded and nonlinear material properties are treated, under assumptions (B2)-(C2). Finally the case of nonlinear anomalies with possibly vanishing material properties is investigated in \Cref{ano_zero}, under assumptions (B3)-(C3). 

\subsection{$\mathbf{\bm{\gamma}_{NL}>\bm{\gamma}_{BG}}$ and bounded}\label{ano_1}
In this Section, $\gamma_{NL}$ and $\gamma_{BG}$ are such that
\begin{equation}
\begin{cases}
\sup_{\Omega\setminus A}\{\gamma_{BG}(x)\} < \underline{\sigma} 
\\0<\underline{\gamma}\leq\gamma_{NL}(x, s)\leq\overline{\gamma}<+\infty &\text{ for a.e. }x\in A, \   \forall s> 0,\\
\end{cases}
\label{eqn:nass0}
\end{equation}
where $\underline{\gamma}$ and $\overline{\gamma}$ are two proper constants. The second relationship in \eqref{eqn:nass0} is the assumption (B1) and, furthermore, $\gamma_{NL}$ satisfy (A1), (A2) and (C1).
In this case, the material property of the anomaly is (i) greater than that of the background and (ii) is upper bounded.
An example of nonlinear material property compatible with conditions~\eqref{eqn:nass0} is shown in \Cref{fig:sigma1}.
\begin{figure}[htp]
\centering
\includegraphics[width=0.75\textwidth]{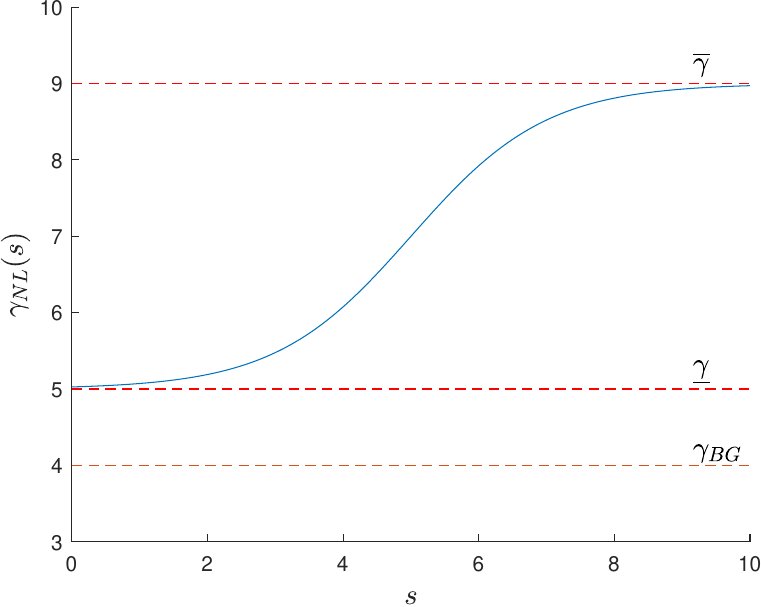}
\caption{$\gamma_{NL}$ compatible with conditions in~\eqref{eqn:nass0}}
\label{fig:sigma1}
\end{figure}

Given $T\subset\Omega$, the test material property $\gamma_T$ is defined as
\begin{equation}
\gamma_T(x)=
\begin{cases}
\gamma_{BG}(x) & \text{in $\Omega\setminus T$}\\
\underline{\gamma} & \text{in $T$}.
\end{cases}
\label{eqn:sigmat}
\end{equation}
The \Cref{fig:targetgeo} shows the unknown anomaly and the three possible cases of (i) the test anomaly $T$ is completely contained into the exterior of the outer support of the unknown anomaly $A$ $\left( T\cap A^*=\emptyset \right)$, (ii) the test anomaly $T$ is partially contained in the outer support of $A$ $\left( T\nsubseteq A^* \right)$ and $\left( T\cap A^*\neq\emptyset \right)$, and (iii) the test anomaly $T$ is completely contained in the outer support of $A$ $\left( T \subseteq  A^* \right)$.
\begin{figure}[htp]
\centering
\subfloat
{\includegraphics[width=.3\textwidth]{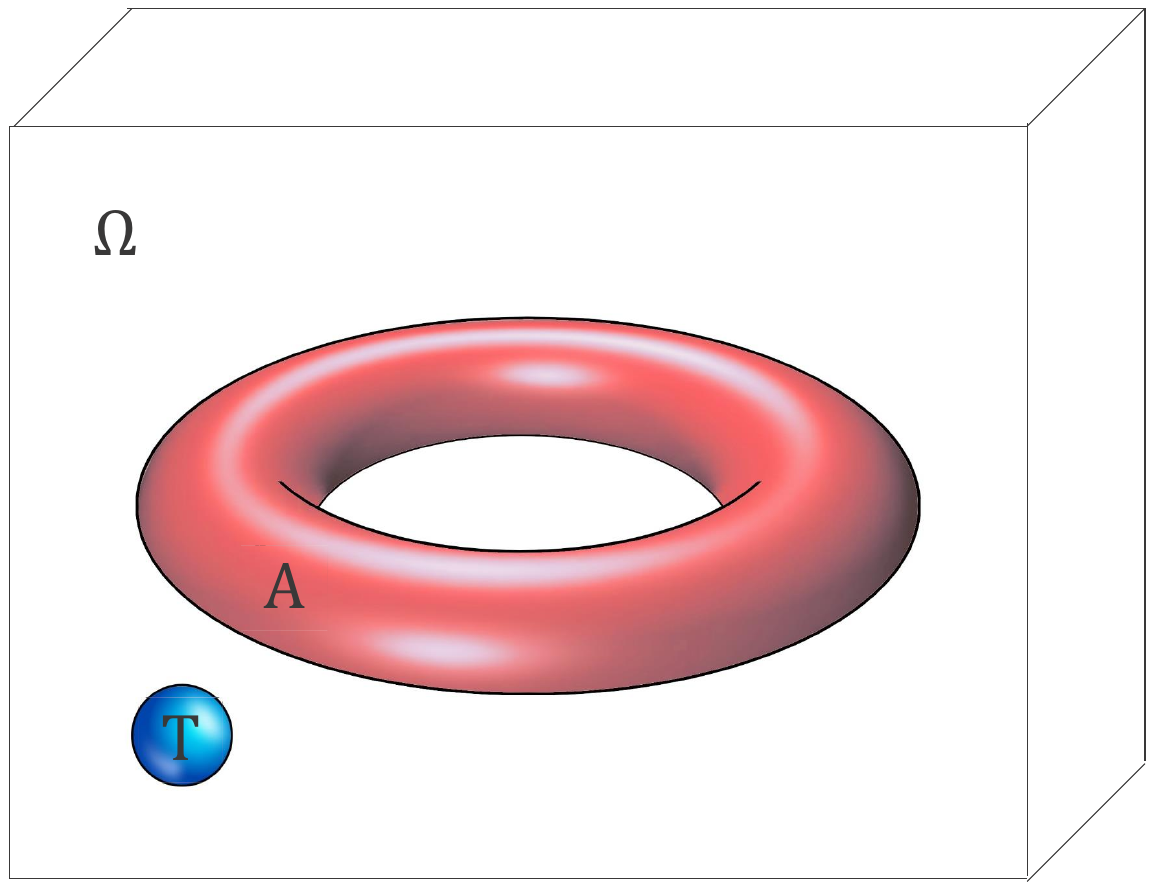}
\label{subfig:test1}}
\subfloat
{\includegraphics[width=.3\textwidth]{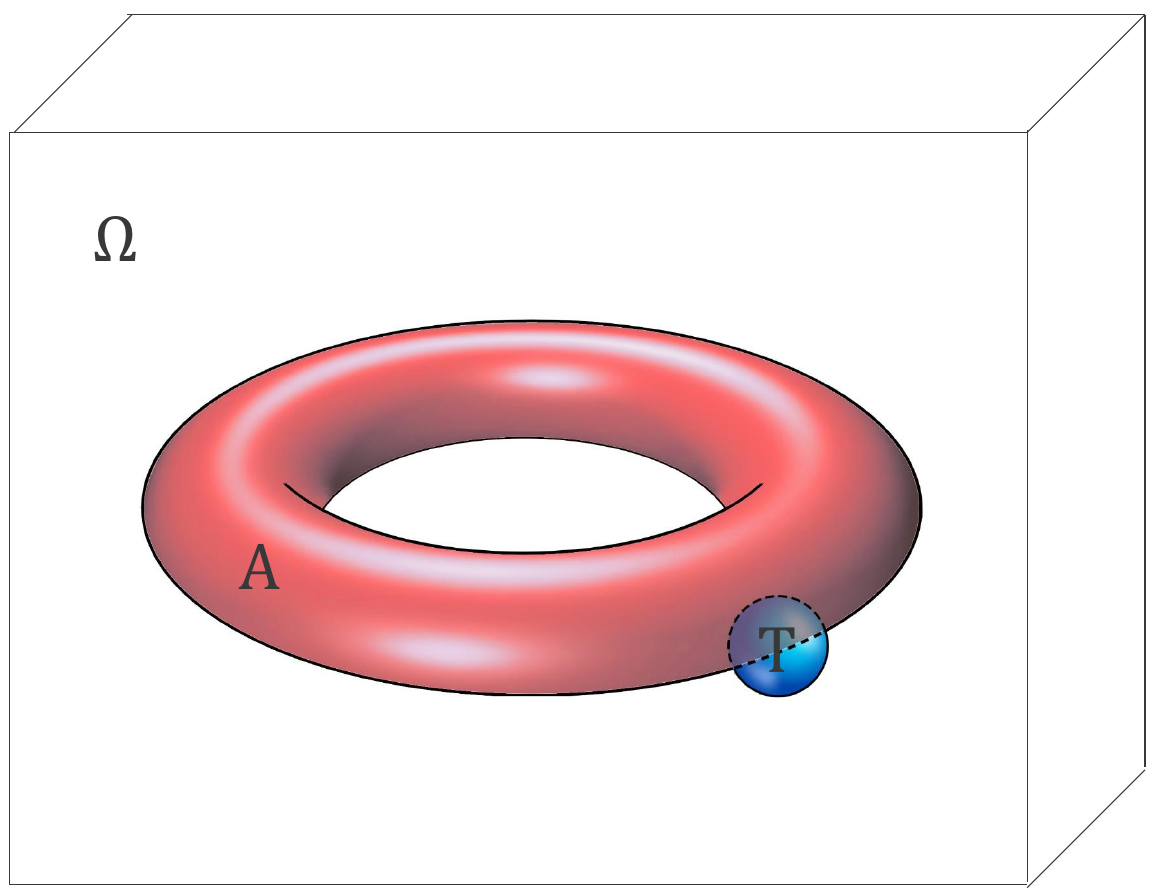}
\label{subfig:test2}}
\subfloat
{\includegraphics[width=.3\textwidth]{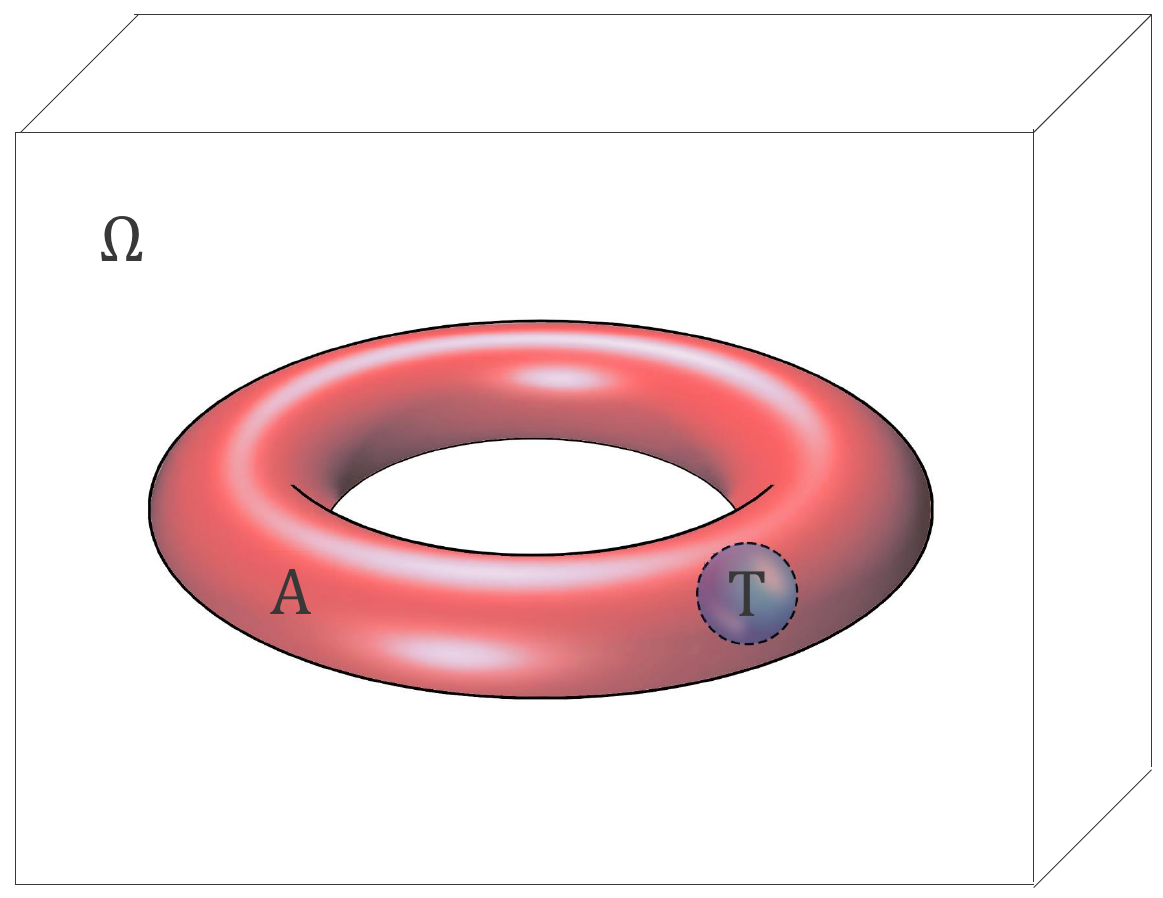}
\label{subfig:test3}} 
\caption{The three reference cases: $T \cap A^* = \emptyset$ 
(left), $T \nsubseteq A^*$ and $T\cap A^*\neq\emptyset$ (center), $T \subseteq A^*$ (right). For the sake of simplicity is assumed that $A=A^*$.}
\label{fig:targetgeo}
\end{figure}

The following Theorem refers to the cases shown in \Cref{fig:targetgeo} (left) and (center). 

\begin{theorem}\label{lem:lem1}
Let $\gamma_{NL}$ satisfy (A1), (A2), (B1) and (C1), and $\gamma_{BG}(x) \in L_+^\infty (\Omega)$ be piecewise analytic such that $\sup_{\Omega\setminus A}\{\gamma_{BG}(x)\} < \underline{\gamma}$. Let the material properties $\gamma_A(x,E)$ and $\gamma_T(x)$ be defined as in \eqref{eqn:sigmaa} and \eqref{eqn:sigmat}, respectively. Then,
\begin{equation}
T \nsubseteq A^* \Longrightarrow \overline{\Lambda}_T\nleqslant\overline{\Lambda}_A\qquad \forall A, T\in\mathcal S (\Omega).
\end{equation}
Moreover, if $A \in \mathcal S (\Omega)$ has a connected complement, then
\begin{equation}
\label{doppia_implicazione}
T\subseteq A \iff \overline{\Lambda}_T\leqslant\overline{\Lambda}_A \qquad\forall\ T\in\mathcal S (\Omega).
\end{equation}
\end{theorem}
\begin{proof}
Let $u_A$ and $u_T$ be the weak solution of~\eqref{P1}, with material property equal to $\gamma_A$ and $\gamma_T$, respectively.

It turns out that
\begin{equation}
    \left\langle \overline{\Lambda}_A(f),f\right\rangle = \int_0^{\abs{\nabla u_A(x)}}\gamma_A(x,\eta)\eta\,d\eta\,dx \le \int_0^{\abs{\nabla u_T(x)}}\gamma_A(x,\eta)\eta\,d\eta\,dx,
\end{equation}
where the first equality comes from \eqref{ADtN=E} combined with \eqref{eqn:Ea}, and the inequality comes from the minimality of $u_A$ (see \eqref{eqn:varprob}).
Hence,
\begin{equation}
\label{eqn:int1}
\langle\overline{\Lambda}_A(f)-\overline{\Lambda}_T(f),f\rangle \leq \int_{\Omega}\int_0^{\abs{\nabla u_T(x)}}\left(\gamma_A(x,\eta)-\gamma_T(x)\right)\eta\,d\eta\,dx,
\end{equation}
where it has been exploited that $\left\langle \overline{\Lambda}_T(f), f\right\rangle=\int_{\Omega}\int_0^{\abs{\nabla u_T(x)}}\gamma_T(x)\eta\,d\eta\,dx$ (see \eqref{ADtN=E} and \eqref{eqn:Ea}, written for $\gamma_T$, rather than $\gamma_A$).

In the following it is assumed that $T\cap A^* \neq \emptyset$. The case when $T\cap A^*$ is empty, can be treated similarly by taking into account that the integrals over $T\cap A^*$ disappear.

By substituting the expressions of $\gamma_A$ and $\gamma_T$ from~\eqref{eqn:sigmaa} and~\eqref{eqn:sigmat} in~\eqref{eqn:int1}, it results
\begin{equation*}
    \begin{split}
\langle\overline{\Lambda}_A(f)-\overline{\Lambda}_T(f),f\rangle
& \leq\int_{A^*\setminus T}\int_0^{\abs{\nabla u_T(x)}}\left(\gamma_{NL}(x,\eta)-\gamma_{BG}(x)\right)\eta\,d\eta\,dx \\
&\quad + \int_{A^*\cap T}\int_0^{\abs{\nabla u_T(x)}}\left(\gamma_{NL}(x,\eta)-\underline{\gamma}\right)\eta\,d\eta\,dx \\
&\quad -\int_{T\setminus A^*}\int_0^{\abs{\nabla u_T(x)}}\left(\underline{\gamma}-\gamma_{BG}(x)\right)\eta\,d\eta\,dx \\
&\leq \int_{A^*\setminus T}\int_0^{\abs{\nabla u_T(x)}}\left(\overline             {\gamma}-\gamma_{BG}(x)\right)\eta\,d\eta\,dx\\
&\quad +\int_{A^*\cap T}\int_0^{\abs{\nabla u_T(x)}}\left(\overline{\gamma}-\underline{\gamma}\right)\eta\,d\eta\,dx \\
& \quad -\int_{T\setminus A^*}\int_0^{\abs{\nabla u_T(x)}}\left(\underline{\gamma}-\gamma_{BG}(x)\right)\eta\,d\eta\,dx \\
    \end{split}
\end{equation*}
and, therefore,
\begin{equation}
\label{chainLambda}
    \begin{split}
 \langle\overline{\Lambda}_A(f)-\overline{\Lambda}_T(f),f\rangle
 & \leq\frac{1}{2} \int_{A^*\setminus T}\left(\overline{\gamma}-\gamma_{BG}(x)\right)\abs{\nabla u_T(x)}^2\,dx\\
&
\quad + \frac{1}{2} \int_{A^*\cap T}\left(\overline{\gamma}-\underline{\gamma}\right)\abs{\nabla u_T(x)}^2\,dx\\
&\quad -\frac{1}{2}\int_{T\setminus A^*}\left(\underline{\gamma}-\gamma_{BG}(x)\right)\abs{\nabla u_T(x)}^2\,dx.
    \end{split}
\end{equation}

Let $B_{\varepsilon}$ be a ball of radius $\varepsilon>0$ contained into the interior of $T\setminus A^*$, and let $U \subseteq \overline{\Omega}$ be a relatively open, connected to $\partial\Omega$ such that $B_{\varepsilon}\subset U$. 
\begin{figure}[htp]
\centering
\includegraphics[width=.4\textwidth]{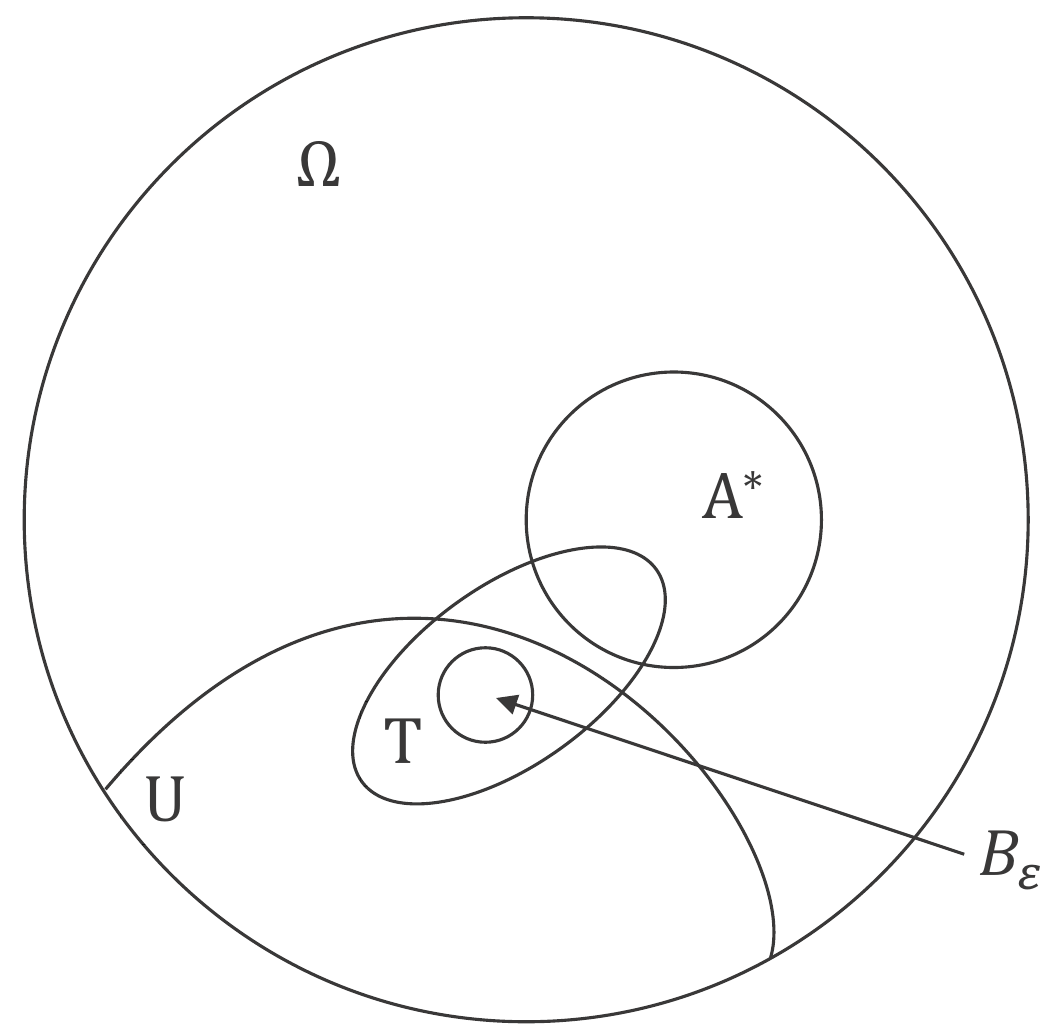}
\caption{Geometric relationships between sets $\Omega$, $U$, $T$, $A^*$ and $B_{\varepsilon}$.}
\label{fig:ta1}
\end{figure}
From Proposition~\ref{lem:bridge}, there exists a sequence of boundary potentials $\{f_n\}_{n\in\mathbb{N}}\subset X_{\diamond}$ such that the sequence of solutions  $\{u_n\}_{n\in\mathbb{N}}$ of problem \eqref{eqn:lemfn}, with $\gamma=\gamma_T$,
have the asymptotic behaviour \eqref{eqn:lemlim} applied to $S_1=\Omega\setminus U$ and $S_2=B_\varepsilon$:
\begin{equation*}
\lim_{n\to+\infty}\int_{\Omega\setminus U}\abs{\nabla u_n(x)}^2\,dx=0 \quad \text{and} \quad \lim_{n\to+\infty}\int_{B_{\varepsilon}}\abs{\nabla u_n(x)}^2\,dx=+\infty.
\end{equation*}
Consequently, it turns out that 
\begin{align}
&\int_{A^*\setminus T}\abs{\nabla u_n(x)}^2\,dx  \to 0, \label{lim_1} \\
& \int_{A^*\cap T}\abs{\nabla u_n(x)}^2\,dx  \to 0, \label{lim_2} \\
& \int_{T\setminus A^*} \abs{\nabla u_n(x)}^2\,dx\geq\int_{B_{\varepsilon}}\abs{\nabla u_n(x)}^2\,dx  \to+\infty.\label{lim_3}
\end{align}
Therefore, by combining \eqref{chainLambda} for $f=f_n$ and $u_T=u_n$, together with \eqref{lim_1}-\eqref{lim_3}, it results that
\begin{equation*}
 \begin{split}
\langle\overline{\Lambda}_A(f_n)-\overline{\Lambda}_T(f_n),f_n\rangle \to-\infty.
 \end{split} 
\end{equation*}

This proves that
\begin{equation}
    \label{mono_suff}
  T\nsubseteq A^* \Longrightarrow\overline{\Lambda}_T\nleqslant\overline{\Lambda}_A.
\end{equation}

When the outer support of $A$ coincides with $A$, i.e. $A^*=A$, equation \eqref{mono_suff} combined with \eqref{eqn:mono2} gives the equivalence stated in \eqref{doppia_implicazione}.
\end{proof}

\subsection{$\mathbf{\bm{\gamma}_{NL}<\bm{\gamma}_{BG}}$ and bounded}\label{ano_2}
In this section $\gamma_{NL}$ and $\gamma_{BG}$ are such that
\begin{equation}
\begin{cases}
\overline{\gamma} < \inf_{\Omega\setminus A}\{\gamma_{BG}(x)\}   &
\\
0<\underline{\gamma}\leq\gamma_{NL}(x,s)\leq\overline{\gamma}<+\infty &\text{ for a.e. }x\in A, \   \forall s> 0,\\
\end{cases}
\label{eqn:nass1}
\end{equation}
where $\underline{\gamma}$ and $\overline{\gamma}$ are two proper constants. The second relationship in \eqref{eqn:nass1} is the assumption (B1) and, furthermore, $\gamma_{NL}$ satisfy (A1), (A2) and (C1).
In this case, the material property of the anomaly is (i) smaller than that of the background and (ii) is lower bounded.

An example of nonlinear material property compatible with conditions~\eqref{eqn:nass1} is shown in \Cref{fig:sigma2}.
\begin{figure}[htp]
\centering
\includegraphics[width=.75\textwidth]{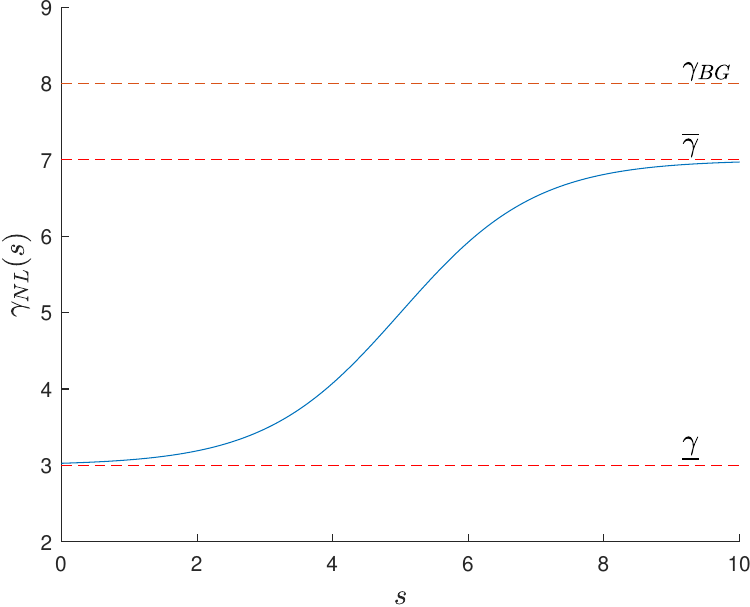}
\caption{$\gamma_{NL}$ compatible with conditions in~\eqref{eqn:nass1}}
\label{fig:sigma2}
\end{figure}

\noindent Given an arbitrary test domain $T\subset\Omega$, the corresponding (test) material property $\gamma_T$ is defined as
\begin{equation}
\gamma_T(x)=
\begin{cases}
\gamma_{BG}(x) & \text{in $\Omega\setminus T$}\\
\overline{\gamma} & \text{in $T$}.
\end{cases}
\label{eqn:sigmat1}
\end{equation}

\begin{theorem}\label{lem:lem2}
Let $\gamma_{NL}$ satisfy (A1), (A2), (B1) and (C1), and $\gamma_{BG}(x) \in L_+^\infty (\Omega)$ be piecewise analytic such that $\overline{\gamma} < \inf_{\Omega\setminus A}\{\gamma_{BG}(x)\}$. Let the material properties $\gamma_A(x,E)$ and $\gamma_T(x)$ be defined as in \eqref{eqn:sigmaa} and \eqref{eqn:sigmat1}, respectively. 
Then,
\begin{equation}
T \nsubseteq A^* \Longrightarrow \overline{\Lambda}_T\nleqslant\overline{\Lambda}_A\qquad \forall A, T\in\mathcal S (\Omega).
\end{equation}
Moreover, if $A\in\mathcal S (\Omega)$ has a connected complement, then for every $T\subset\Omega$,
\begin{equation}
\label{doppia_implicazione1}
T\subseteq A \iff \overline{\Lambda}_T\leqslant\overline{\Lambda}_A \qquad\forall\ T\in\mathcal S (\Omega).
\end{equation}
\end{theorem}
\begin{proof}

Let $\gamma_A^I$ be the material property defined as
\begin{equation}
\label{eqn:sigmaai}
\gamma_A^I(x)=
\begin{cases}
\gamma_{BG}(x) & \text{in $\Omega\setminus A$}\\
\underline{\gamma} & \text{in A}.
\end{cases}
\end{equation}
Since $\gamma_A^I \le \gamma_A$, the Monotonicity Principle \cite{MPpiecewise,art:Co21} implies that $\overline{\Lambda}_A^I\leqslant\overline{\Lambda}_A$, being $\overline{\Lambda}_A^I$ the DtN operator related to $\gamma_A^I$. By combining this latter inequality with \eqref{eqn:int1}, it follows that 
\begin{equation}\label{eqn:int3}
\langle \overline{\Lambda}_T(f)-\overline{\Lambda}_A(f),f\rangle\leq\langle \overline{\Lambda}_T(f)-\overline{\Lambda}_A^I(f),f\rangle\leq 
\int_{\Omega}\int_0^{\abs{\nabla u_A^I(x)}}(\gamma_T(x)-\gamma_A^I(x))\,d\eta\,dx,
\end{equation}
where $u_A^I$ is the solution of 
\begin{equation*}
\begin{cases}
\nabla\cdot\left(\gamma_A^I(x)\nabla u^{I}_A(x)\right)=0 & \text{ in } \Omega\\
u^{I}_A(x)=f(x) & \text{ on }\partial \Omega. \\
\end{cases}
\end{equation*}

In the following it is assumed that $T\cap A^* \neq \emptyset$. The case when $T\cap A^*$ is empty, can be treated similarly by taking into account that the integrals over $T\cap A^*$ disappear. 

For $T\cap A^*\neq\emptyset$, by replacing $\gamma_T$ and $\gamma_A^I$ with their expressions (see \eqref{eqn:sigmat1} and \eqref{eqn:sigmaai}), it results that
\begin{equation}\label{eqn:chainLambda_2}
    \begin{split}
\langle\overline{\Lambda}_T(f)-\overline{\Lambda}_A(f),f\rangle
& \leq\int_{A^*\setminus T}\left(\gamma_{BG}(x)-\underline{\gamma}\right)\abs{\nabla u_A^I(x)}^2\,dx \\
&\quad + \int_{A^*\cap T}\left(\overline{\gamma}-\underline{\gamma}\right)\abs{\nabla u_A^I(x)}^2\,dx \\
&\quad -\int_{T\setminus A^*}\left(\overline{\gamma}-\gamma_{BG}(x)\right)\abs{\nabla u_A^I(x)}^2\,dx. 
\end{split}
\end{equation}

Let $B_{\varepsilon}$ be a ball of radius $\varepsilon > 0$ contained into the interior of $T\setminus A^*$, and let $U \subseteq \overline{\Omega}$ be a relatively open set, connected to $\partial\Omega$ such that $B_{\varepsilon}\subset U$ (see \Cref{fig:ta1}). 

From Proposition~\ref{lem:bridge}, there exists a sequence of boundary potentials $\{f_n\}_{n\in\mathbb{N}}\subset X_{\diamond}$ such that the sequence of solutions  $\{u_n\}_{n\in\mathbb{N}}$ of problem \eqref{eqn:lemfn}, with $\gamma=\gamma_A^I$,
have the asymptotic behaviour \eqref{eqn:lemlim} applied to $S_1=\Omega\setminus U$ and $S_2=B_\varepsilon$:
\begin{equation*}
\lim_{n\to+\infty}\int_{\Omega\setminus U}\abs{\nabla u_n(x)}^2\,dx=0 \quad \text{and} \quad \lim_{n\to+\infty}\int_{B_{\varepsilon}}\abs{\nabla u_n(x)}^2\,dx=+\infty.
\end{equation*}

Consequently, it turns out that 
\begin{align}
&\int_{A^*\setminus T}\abs{\nabla u_n(x)}^2\,dx  \to 0, \label{lim_1_I} \\
& \int_{A^*\cap T}\abs{\nabla u_n(x)}^2\,dx  \to 0, \label{lim_2_I} \\
& \int_{T\setminus A^*} \abs{\nabla u_n(x)}^2\,dx\geq\int_{B_{\varepsilon}}\abs{\nabla u_n(x)}^2\,dx  \to+\infty.\label{lim_3_I}
\end{align}
Therefore, by combining \eqref{eqn:chainLambda_2} for $f=f_n$ and $u_A^I=u_n$, together with \eqref{lim_1_I}-\eqref{lim_2_I}-\eqref{lim_3_I}, it results that
\begin{equation*}
 \begin{split}
\langle\overline{\Lambda}_T(f_n)-\overline{\Lambda}_A(f_n),f_n\rangle\to-\infty.
 \end{split} 
\end{equation*}

This proves that
\begin{equation}
    \label{mono_suff_I}
  T\nsubseteq A^* \Longrightarrow\overline{\Lambda}_T\nleqslant\overline{\Lambda}_A.
\end{equation}

When the outer support of $A$ coincides with $A$, i.e. $A^*=A$, equation \eqref{mono_suff_I} combined with \eqref{eqn:mono2} gives the equivalence stated in \eqref{doppia_implicazione1}.
\end{proof}

\subsection{$\mathbf{\bm{\gamma}_{NL}>\bm{\gamma}_{BG}}$ and unbounded}\label{ano_infty}
In this section, $\gamma_{NL}$ satisfies (A1), (A2), (B2), (C2). For the convenience of the reader, it is recalled that accordingly to (B2), for fixed $1<q<+\infty$, there exist three positive constants $\underline{\gamma}\le\overline{\gamma}$ and $s_0$ such that
\begin{equation} \label{eqn:nass2_0}
\underline{\gamma} \leq\gamma_{NL}(x, s)\leq 
\begin{cases}
    \overline{\gamma}\left[1+\left( \frac{s}{s_0} \right)^{q-2}\right] & \text{if}\ q\ge 2,\\
    \overline{\gamma}\left( \frac{s}{s_0} \right)^{q-2} &  \text{if}\ 1<q< 2,
\end{cases}
\end{equation}
$\text{for a.e.}\ x\in {\overline A}\ \text{and}\ \forall s>0$.
Furthermore, $\gamma_{BG}$ fulfills
\begin{equation}
\sup_{\Omega\setminus A}\{\gamma_{BG}(x)\} < \underline{\gamma} 
\\
\label{eqn:nass2}
\end{equation}
In this case, the material property of the anomaly is (i) greater than that of the background and (ii) may be not upper bounded. An example of nonlinear material property compatible with conditions \eqref{eqn:nass2_0} and \eqref{eqn:nass2} is shown in \Cref{fig:sigma3}.
\begin{figure}[ht]
\centering
\includegraphics[width=.75\textwidth]{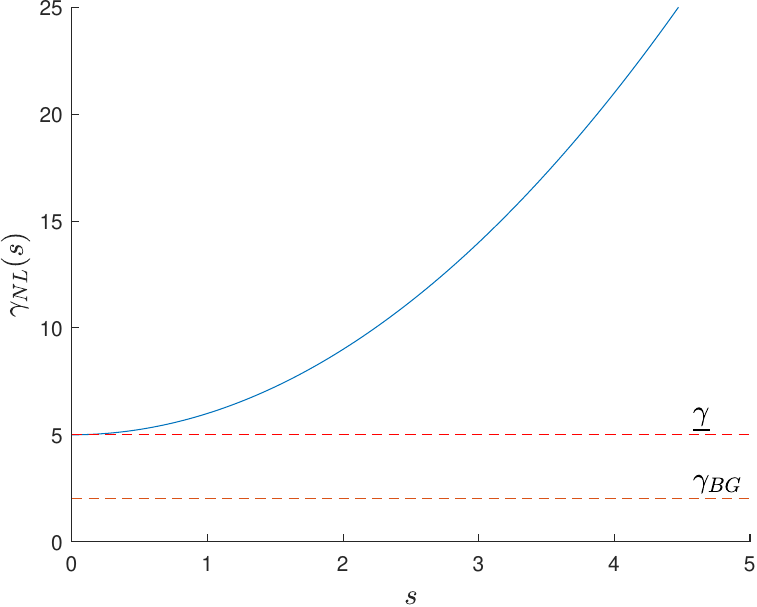}
\caption{$\gamma_{NL}$ compatible with conditions in~\eqref{eqn:nass2_0} and \eqref{eqn:nass2}}
\label{fig:sigma3}
\end{figure}

\begin{remark}\label{rem:ass_inf}
Assumption \eqref{eqn:nass2_0} bounds the growth rate of the material property. The case $q \geq 2$ corresponds to unbounded material properties increasing with the amplitude of $s$ ($\gamma\to +\infty$, for $s\to +\infty$). The case $q<2$ corresponds to unbounded material properties decreasing with the amplitude of $s$ ($\gamma\to+\infty$, for $s\to 0^+$). 

These assumptions cover almost all cases arising from applications (as superconductors or composite in field grading applications, see references in \Cref{sec:int}).
\end{remark}

Firstly, some definitions and lemmas are introduced. Specifically, the \emph{average} DtN related to the material property
\begin{equation}\label{eqn:sigmainf}
\gamma_A^{\infty}(x)=
\begin{cases}
\gamma_{BG}(x) & \text{in $\Omega\setminus A$}\\
+\infty & \text{in $A$}
.
\end{cases}
\end{equation}
is denoted as $\overline{\Lambda}^{\infty}_A$. The solution $u^{\infty}_A \in H^1(\Omega)$ related to $\gamma^\infty_A$ solves the problem
\begin{equation}\label{u_infty}
\begin{cases}
\nabla\cdot\left(\gamma_{BG}(x)\nabla u^{\infty}_A(x)\right)=0 & \text{in $\Omega\setminus A$} \\
\nabla u^{\infty}_A(x)= 0 & \text{in $A$}\\
u^{\infty}_A(x)=f(x) & \text{on $\partial \Omega$}. \\
\end{cases}
\end{equation}

The following Lemmas (see Appendix \ref{sec:appa}) provide crucial inequalities for treating inclusions with an infinite value for the material property. In the case of ERT, it is worth noting that the perfectly conducting inclusions naturally appear in nonlinear problems where the boundary data is either large or small enough (see \cite{corboesposito2023thep0laplacesignature,corboesposito2023theplaplacesignature}).
\begin{lemma}
\label{lem:moninf}
Let $\overline{\Lambda}_A$ and $\overline{\Lambda}_A^{\infty}$ be the \emph{average} DtN operators corresponding to the material properties $\gamma_A(x,E)$ and $\gamma_A^\infty(x)$, defined in \eqref{eqn:sigmaa} and \eqref{eqn:sigmainf}, respectively, under assumptions~\eqref{eqn:nass2}. Then, it results that
\begin{displaymath}
  \overline{\Lambda}_A\leqslant \overline{\Lambda}_A^{\infty}.
\end{displaymath}
\end{lemma}
\begin{lemma}
\label{lem:ga}
Let $\Lambda_{BG}$ and $\Lambda_A^{\infty}$  be the DtN operators related to the material properties $\gamma_{BG}(x) \in L_+^\infty(\Omega)$ and $\gamma_A^\infty(x)$ defined in \eqref{eqn:sigmainf}. Assuming $\gamma_{BG}(x)$ piecewise analytic and condition~\eqref{eqn:nass2}, there exists a positive constant $K$ such that
\begin{equation*}
0\leq\langle \Lambda_A^{\infty}(f)-\Lambda_{BG}(f), f\rangle\leq K\int_A\abs{\nabla u_{BG}(x)}^2\,dx
\qquad\forall f \in X_\diamond,
\end{equation*}
with $u_{BG}$ solution of 
\begin{equation}
\label{u_BG}
\begin{cases}
\nabla\cdot\left(\gamma_{BG}(x)\nabla u_{BG}(x)\right)=0 & \text{ in } \Omega\\
u_{BG}(x)=f(x) & \text{ on }\partial \Omega. \\
\end{cases}
\end{equation}
\end{lemma}

It is worth noting that the original problem \eqref{P1} reduces problem \eqref{u_BG} when there are no anomalies ($A=\emptyset$).

The proofs of Lemma \ref{lem:moninf} and Lemma \ref{lem:ga} are provided in Appendix~\ref{sec:appa}. An inequality similar to that of Lemma \ref{lem:ga} is available for the Neumann-to-Dirichlet operator in~\cite{art:Ga20}.

The following Theorem holds.
\begin{theorem}\label{lem:lem3}
Let $\gamma_{NL}$ satisfy (A1), (A2), (B2) and (C2), and $\gamma_{BG} \in L_+^\infty (\Omega)$ be piecewise analytic such that $\sup_{\Omega\setminus A}\{\gamma_{BG}(x)\} < \underline{\gamma}$. Let the material properties $\gamma_A(x,E)$ and $\gamma_T(x)$ be defined as in \eqref{eqn:sigmaa} and \eqref{eqn:sigmat}, respectively
. Then,
\begin{equation}
T \nsubseteq A^* \Longrightarrow \overline{\Lambda}_T\nleqslant\overline{\Lambda}_A\qquad \forall A, T\in\mathcal S (\Omega).
\end{equation}
Moreover, if $A\in\mathcal S (\Omega)$ has a connected complement, then for every $T \in \mathcal{S}(\Omega)$,
\begin{equation}
\label{doppia_implicazione2}
T\subseteq A \iff \overline{\Lambda}_T\leqslant\overline{\Lambda}_A \qquad\forall\ T\in\mathcal S (\Omega).
\end{equation}
\end{theorem}
\begin{proof}
From Lemma \ref{lem:moninf}, it follows that
\begin{equation}
\label{eqn:infin1}
\begin{split}
    \langle\overline{\Lambda}_A(f)-\overline{\Lambda}_T(f),f\rangle&\leq \langle \overline{\Lambda}_A^{\infty}(f)-\overline{\Lambda}_T(f),f\rangle \\
    &= \langle\overline{\Lambda}_A^{\infty}(f)-\overline{\Lambda}_{BG}(f),f\rangle-\langle \overline{\Lambda}_{T}(f)-\overline{\Lambda}_{BG}(f),f\rangle,
\end{split}
\end{equation}
with $\overline{\Lambda}_{BG}$ being the \emph{average} DtN corresponding to material property $\gamma_{BG}$.

The first term at the r.h.s. of \eqref{eqn:infin1} can be upper bounded as follows
\begin{equation}\label{eqn:infin2}
\langle\overline{\Lambda}_A^{\infty}(f)-\overline{\Lambda}_{BG}(f),f\rangle=\frac{1}{2}\langle \Lambda_A^{\infty}(f)-\Lambda_{BG}(f),f\rangle\leq K_1 \int_{A}\abs{\nabla u_{BG}(x)}^2\,dx,
\end{equation}
where the equality holds since both $\overline{\Lambda}_{BG}$ and $\overline{\Lambda}_A^{\infty}$ are associated to linear material properties, whereas the inequality follows Lemma \ref{lem:ga}.

The second term at the r.h.s. of \eqref{eqn:infin1}, can be lower bounded as follows
\begin{equation}
\label{eqn:ikin}
\begin{split}
\langle\overline{\Lambda}_{T}(f)-\overline{\Lambda}_{BG}(f),f \rangle  & =\frac{1}{2}\langle \Lambda_{T}(f)-\Lambda_{BG}(f),f \rangle     \\
  &  \ge \frac{1}{2}\int_\Omega\frac{\gamma_{BG}(x)}{\gamma_T(x)}(\gamma_T(x)-\gamma_{BG}(x))\abs{\nabla u_{BG}(x)}^2\,dx \\
   & = \frac{1}{2}\int_T\frac{\gamma_{BG}(x)}{\underline{\gamma}}(\underline{\gamma}-\gamma_{BG}(x))\abs{\nabla u_{BG}(x)}^2\,dx,
\end{split}
\end{equation}
where the first line holds because both $\overline{\Lambda}_{BG}$ and $\overline{\Lambda}_A^{\infty}$ are associated to linear material properties, the second line comes from \cite[Lemma 2.1]{art:Bra17} for $p = 2$ (see also references therein) and, finally, the third line follows from the fact that $\gamma_T$ and $\gamma_{BG}$ agree on $\Omega\setminus T$.

By combining~\eqref{eqn:infin1},~\eqref{eqn:infin2} and~\eqref{eqn:ikin}, it results that
\begin{equation}\label{eqn:dis_PEC}
\begin{split}
    \left\langle \left(\overline{\Lambda}_A(f)-\overline{\Lambda}_T(f)\right),f\right\rangle 
    &\leq K_1\int_{A}\abs{\nabla u_{BG}(x)}^2\,dx\\
    &\quad -\frac{1}{2}\int_T\frac{\gamma_{BG}(x)}{\underline{\gamma}}(\underline{\gamma}-\gamma_{BG}(x))\abs{\nabla u_{BG}(x)}^2\,dx.
\end{split}
\end{equation}

Let $B_{\varepsilon}$ be a ball of radius $\varepsilon>0$ contained into the interior of $T\setminus A^*$, and let $U \subseteq \overline{\Omega}$ be a relatively open, connected to $\partial\Omega$ such that $B_{\varepsilon}\subset U$. 
\\ From Proposition~\ref{lem:bridge}, there exists a sequence of boundary potentials $\{f_n\}_{n\in\mathbb{N}}\subset X_{\diamond}$ such that the sequence of solutions  $\{u_n\}_{n\in\mathbb{N}}$ of problem \eqref{eqn:lemfn}, with $\gamma=\gamma_{BG}$, $S_1=\Omega\setminus U$, and $S_2=B_\varepsilon$,
has the asymptotic behaviour:
\begin{equation*}
\lim_{n\to+\infty}\int_{\Omega\setminus U}\abs{\nabla u_n(x)}^2\,dx=0 \quad \text{and} \quad \lim_{n\to+\infty}\int_{B_{\varepsilon}}\abs{\nabla u_n(x)}^2\,dx=+\infty.
\end{equation*}

Consequently, it turns out that 
\begin{align}
&\int_{A}\abs{\nabla u_n(x)}^2\,dx \le \int_{\Omega \setminus U}\abs{\nabla u_n(x)}^2\,dx \to 0, \label{lim_1_PEC} \\
& \int_{T} \abs{\nabla u_n(x)}^2\,dx\geq\int_{B_{\varepsilon}}\abs{\nabla u_n(x)}^2\,dx  \to+\infty.\label{lim_2_PEC}
\end{align}
Therefore, by combining \eqref{eqn:dis_PEC} for $f=f_n$ and $u_{BG}=u_n$, together with \eqref{lim_1_PEC} and \eqref{lim_2_PEC}, it results that
\begin{equation*}
\left\langle \left(\overline{\Lambda}_A(f_n)-\overline{\Lambda}_T(f_n)\right),f_n\right\rangle 
\to-\infty.
\end{equation*}

This proves that
\begin{equation}
    \label{mono_suff_PEC}
  T\nsubseteq A^* \Longrightarrow\overline{\Lambda}_T\nleqslant\overline{\Lambda}_A.
\end{equation}
When the outer support of $A$ coincides with $A$, i.e. $A^*=A$, equation \eqref{mono_suff_PEC} combined with \eqref{eqn:mono2} gives the equivalence stated in \eqref{doppia_implicazione2}.
\end{proof}

\subsection{$\mathbf{\bm{\gamma}_{NL}<\bm{\gamma}_{BG}}$ and possibly vanishing}\label{ano_zero}
In this section, $\gamma_{NL}$ and $\gamma_{BG}$ are such that
\begin{equation}
\begin{cases}
\overline{\gamma} <\inf_{\Omega\setminus A}\{
\gamma_{BG}(x) \}  &\text{ in } \Omega\setminus A\\
\underline{\gamma} \left(\frac{s}{s_0}\right)^{q-2}\leq \gamma_{NL}(x, s)\leq \overline{\gamma} &\text{ for a.e. }x\in A, \   \forall s> 0,
\end{cases}
\label{eqn:nass3}
\end{equation}
where $q\in [2,\infty)$ and $\overline{\gamma}$, $\underline{\gamma}$, $s_0$ are three proper constants. The second relationship in \eqref{eqn:nass3} is the assumption (B3) and, furthermore, $\gamma_{NL}$ satisfy (A1), (A2) and (C3).

In this case, the material property of the anomaly is (i) smaller than that of the background and (ii) may be vanishing. An example of nonlinear material property compatible with conditions~\eqref{eqn:nass3} is shown in \Cref{fig:sigma4}.

\begin{figure}[ht]
\centering
\includegraphics[width=.75\textwidth]{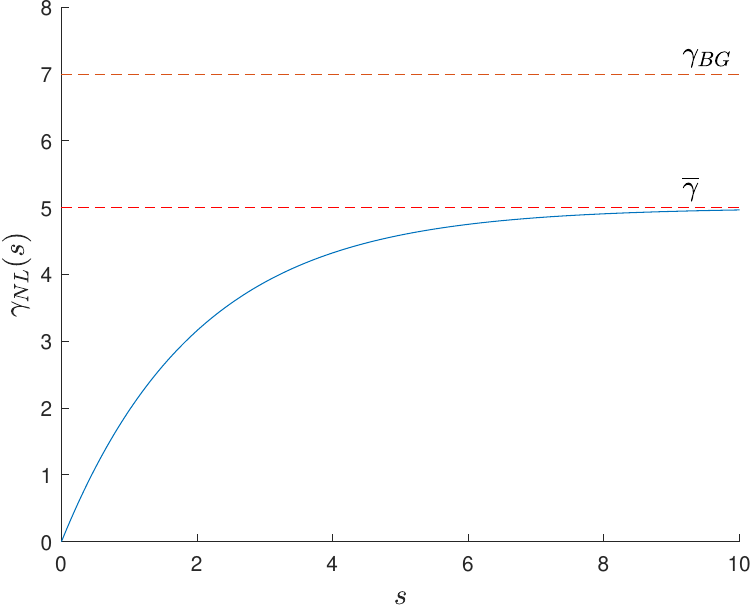}
\caption{$\gamma_{NL}$ compatible with conditions in~\eqref{eqn:nass3}}
\label{fig:sigma4}
\end{figure}


In this case, let $\overline{\Lambda}_A^0$ be the \emph{average} DtN associated to the material property
\begin{equation}\label{eqn:sigma0}
\gamma_A^0(x)=
\begin{cases}
\gamma_{BG}(x) & \text{ in }\Omega\setminus A\\
0 & \text{ in }A.
\end{cases}
\end{equation}
Let $u_A^0\in H^1(\Omega)$ be the solution of
\begin{equation*}
\begin{cases}
    \nabla\cdot(\gamma_A^0(x)\nabla u_A^0(x))=0  & \text{in }\Omega \\
    u_A^0(x)=f(x) & \text{on }\partial\Omega
\end{cases}
\end{equation*}
Solution $u_A^0$, when restricted to $\Omega \setminus A$, solves 
\begin{equation}\label{u_zero}
\begin{cases}
\nabla\cdot\left(\gamma_{BG}(x)\nabla u_A^0(x)\right)=0 & \text{in $\Omega\setminus A$} \\
\gamma_{BG}(x)\partial_\nu u_A^0(x)=0 & \text{on $\partial A$} \\
u_A^0(x)=f(x) & \text{on $\partial\Omega$},
\end{cases}
\end{equation}
whereas $u_A^0$ restricted to $A$ solves 
\begin{equation}
\label{u_zero_cascata}
    \begin{cases}
        \Delta u_A^0(x)=0 & \text{in $A$} \\
        u_A^0(x)=g(x) & \text{on $\partial A$},
    \end{cases}
\end{equation}
being $g$ the restriction to $\partial A$ of the solution $u_A^0$ of problem \eqref{u_zero}.

\begin{remark}
Considering ERT, from the physical standpoint, it is reminded that the scalar potential $u_A^0$ represents the electric field, via its gradient. System  \eqref{u_zero} corresponds to a steady current problem in $\Omega \setminus A$, being $A$ impenetrable (non conducting) to the electrical current density. System \eqref{u_zero_cascata} corresponds to an electrostatic problem in $A$. This latter system of PDEs requires the knowledge on $\partial A$ of the solution $u_A^0$ arising from \eqref{u_zero}.    
\end{remark}

The following Lemmas (see Appendix \ref{sec:appb} for the proofs) provide crucial inequalities for treating anomalies with zero value for the material property. In the case of ERT, it is worth noting that the perfectly insulating anomalies naturally appear
in nonlinear problems where the boundary data is either large or small enough (see \cite{corboesposito2023thep0laplacesignature,corboesposito2023theplaplacesignature}).

\begin{lemma}\label{lem:monnul}
Let $\overline{\Lambda}_A$ and $\overline{\Lambda}_A^{0}$ be the \emph{average} DtN operators corresponding to the material properties $\gamma_A(x,s)$ and $\gamma_A^0(x)$, defined in \eqref{eqn:sigmaa} and \eqref{eqn:sigma0}, respectively, under assumptions~\eqref{eqn:nass3}. Then, it results that
\begin{displaymath}
  \overline{\Lambda}_A\geqslant \overline{\Lambda}_A^0.
\end{displaymath}
\end{lemma}
\begin{lemma} \label{lem:locpei}
Let $S_1, S_2 \subset\Omega$ be two open sets such that $\Omega \setminus \overline{S_1}$ is connected, $\partial S_1$ consists of a single connected component, and $S_2\subset \subset \Omega\setminus S_1^*$. Let the linear material property $\gamma\in L^{\infty}_+(\Omega)$ be piece-wise analytic. 
Then there exists a sequence $\{f_n\}_{n\in\mathbb{N}}\subset X_{\diamond}$ of boundary potentials such that the corresponding family of solutions $\{u_n\}_{n\in\mathbb{N}}$ fulfills
\begin{displaymath}
\lim_{n\to+\infty}\int_{S_1}\gamma(x)\abs{\nabla u_n(x)}^2\,dx=0\quad\text{and}\quad
\lim_{n\to+\infty}\int_{S_2}\gamma(x)\abs{\nabla u_n(x)}^2\,dx=+\infty,
\end{displaymath}
where $u_n$ is obtained with the material property $\gamma$ in $\Omega \setminus S_1$ and a vanishing material property in $S_1$, i.e. $\gamma=0$ in $S_1$ and $u_n$ restricted to $\Omega \setminus S_1$, solves
\begin{equation}\label{u_zero*}
\begin{cases}
\nabla\cdot\left(\gamma(x)\nabla u_n(x)\right)=0 & \text{in $\Omega\setminus S_1$} \\
\gamma(x)\partial_\nu u_n(x)=0 & \text{on $\partial S_1$} \\
u_n(x)=f_n(x) & \text{on $\partial\Omega$},
\end{cases}
\end{equation}
whereas $\nu$ is the outer normal to $S_1$ and $u_n$ restricted to $S_1$ solves
\begin{equation}\label{u_zero*_cascata}
    \begin{cases}
        \Delta u_n(x)=0 & \text{in $S_1$}\\
        u_n(x)=g_n(x) & \text{on $\partial S_1$},
    \end{cases}
\end{equation}
being $g_n$ the restriction to $\partial S_1$ of the solution $u_n$ of problem \eqref{u_zero*}.
\end{lemma}

\begin{remark}
Lemma \ref{lem:locpei} can be generalized for $\partial S_1$ consisting of multiple connected components and $S_1$ made by either a single or multiple connected components.

Without loss of generality, consider the steady state current problem, i.e. $\gamma=\sigma$ and $S_1$ a perfect insulating anomaly.
Specifically, let $S_1$ be equal to    $S_1=C_1\cup C_2 \cup \dots\cup C_m$, being $m \ge 1$, $C_1, \dots, C_m$ connected and disjoint. If $C_i=C_i^*$, for all $i=1,\dots,m$, then Lemma \ref{lem:locpei} holds without any modifications. If there exists $C_i \neq C_i^*$, then that there exists a connected component $B$ of $\Omega \setminus S_1$ that is not electrically connected to $\partial \Omega$. In this subset the electric field and the electrical current density are vanishing, therefore, it results that $\sigma(x)\partial_\nu u_n=0$, on $\partial B$.
\end{remark}

Here, our main theorem regarding the Converse of Monotonicity Principle is stated, for the case of material property of the anomaly smaller than the one of background. Anomalies, may also have a vanishing material property, i.e. $\gamma_{NL}(s,x)=0$ in $A$, for some $s$.

\begin{theorem}\label{lem:lem4}
Let $\gamma_{NL}$ satisfy (A1), (A2), (B3) and (C3), and $\gamma_{BG} \in L_+^\infty (\Omega)$ be piecewise analytic such that $\sup_{\Omega\setminus A}\{\gamma_{BG}(x)\} > \overline{\gamma}$. Let the material properties $\gamma_A(x,s)$ and $\gamma_T(x)$ be defined as in \eqref{eqn:sigmaa} and \eqref{eqn:sigmat1}, respectively.
Then,
\begin{equation}
T \nsubseteq A^* \Longrightarrow \overline{\Lambda}_T\nleqslant\overline{\Lambda}_A\qquad \forall A, T\in\mathcal S (\Omega).
\end{equation}
Moreover, if $A\in\mathcal S (\Omega)$ has a connected complement, then for every $T\subset\Omega$,
\begin{equation}
\label{doppia_implicazione3}
T\subseteq A \iff \overline{\Lambda}_T\leqslant\overline{\Lambda}_A \qquad\forall\ T\in\mathcal S (\Omega).
\end{equation}
\end{theorem}

\begin{proof}
In the following it is assumed that $T\cap A^* \neq \emptyset$. The case when $T\cap A^*$ is empty, can be treated similarly by deleting any integrals over $T\cap A^*$. 

First of all, it is worth noting that
\begin{equation}\label{eqn:in_PEI_2}
\begin{split}
2\langle \overline{\Lambda}_T(f)-\overline{\Lambda}_A(f),f\rangle & \leq   \left\langle {\Lambda}_T(f)-{\Lambda}_A^0(f),f\right\rangle\\
& = \int_{\Omega} \gamma_T(x) \abs{\nabla u_T (x)}^2 \, dx - \int_{\Omega} \gamma_A^0(x) \abs{\nabla u_A^0 (x)}^2 \, dx\\
& \leq \int_{\Omega} \left( \gamma_T(x) - \gamma_A^0(x) \right)\abs{\nabla u_A^0 (x)}^2 \, dx \\
& \leq \int_{A^*\setminus T}\gamma_{BG}(x)\abs{\nabla u^0_A(x)}^2\,dx + \int_{A^*\cap T}\overline{\gamma}\abs{\nabla u^0_A(x)}^2\,dx\\
&\quad -\int_{T\setminus A^*}\left(\gamma_{BG}(x)-\overline{\gamma}\right)\abs{\nabla u_A^0(x)}^2\,dx\qquad\forall f \in X_\diamond.
\end{split}
\end{equation}
where in the first line it is exploited \Cref{lem:monnul} and that the average DtN is one half of the \lq\lq classical\rq\rq \ DtN for linear materials (see \eqref{P4}), in the second line \eqref{w-DtN-f} written for both $\gamma_T$ and $\gamma_A^0$ is used, in the third line the minimality of the Dirichlet Energy (see \Cref{MathModel} for a material property equal to $\gamma_T$, rather than $\gamma_A$) is used, in the last line it is exploited that
\begin{equation}
\gamma_T(x) - \gamma_A^0(x) \le\left \{ 
\begin{aligned}
& 0 & & \text{ in } {\Omega \setminus \left( T \cup A^* \right)} \\
&\gamma_{BG}(x) & & \text{ in } { A^* \setminus T} \\
& \overline\gamma & & \text{ in } {A^* \cap T} \\
&  \overline\gamma  -\gamma_{BG} (x)& & \text{ in } {T \setminus A^*}.
\end{aligned}
\right.
\end{equation}

Let $B_{\varepsilon}$ be a spherical neighborhood contained into the interior of $T\setminus A^*$, and let $U \subseteq \overline{\Omega}$ be a relatively open, connected to $\partial\Omega$ such that $B_{\varepsilon}\subset U$ (see \Cref{fig:ta1}).
From Lemma~\ref{lem:locpei}, it follows that there exists a sequence of boundary potentials $\{f_n\}_{n\in\mathbb{N}}\subset X_{\diamond}$ such that the sequence of extended solutions $\{u_n\}_{n\in\mathbb N}$ of problems \eqref{u_zero*} (for $\gamma=\gamma_{BG}$ in $\Omega\setminus S_1$) and \eqref{u_zero*_cascata} has the asymptotic behaviour:
\begin{equation*}
\lim_{n\to+\infty}\int_{\Omega\setminus U}\abs{\nabla u_n(x)}^2\,dx=0 \quad \text{and} \quad \lim_{n\to+\infty}\int_{B_{\varepsilon}}\abs{\nabla u_n(x)}^2\,dx=+\infty.
\end{equation*}
Consequently, it turns out that 
\begin{align}
&\int_{A^*\setminus T}\abs{\nabla u_n(x)}^2\,dx  \to 0, \label{lim_1_PEI} \\
& \int_{A^*\cap T}\abs{\nabla u_n(x)}^2\,dx  \to 0, \label{lim_2_PEI} \\
& \int_{T\setminus A^*} \abs{\nabla u_n(x)}^2\,dx\geq\int_{B_{\varepsilon}}\abs{\nabla u_n(x)}^2\,dx  \to+\infty.\label{lim_3_PEI}
\end{align}
Therefore, by combining \eqref{eqn:in_PEI_2} for $f=f_n$ and $u_A^0=u_n$, together with \eqref{lim_1_PEI}-\eqref{lim_3_PEI}, it results that
\begin{equation*}
\langle\overline{\Lambda}_A(f_n)-\overline{\Lambda}_T(f_n),f_n\rangle\to-\infty.
\end{equation*}

This proves that
\begin{equation}
    \label{mono_suff_PEI}
  T\nsubseteq A^* \Longrightarrow\overline{\Lambda}_T\nleqslant\overline{\Lambda}_A.
\end{equation}

When the outer support of $A$ coincides with $A$, i.e. $A^*=A$, equation \eqref{mono_suff_PEI} combined with \eqref{eqn:mono2} gives the equivalence stated in \eqref{doppia_implicazione3}.
\end{proof} 

\section{Stability of the Imaging Method}\label{sec:stab}
In this section, the intrinsic stability of the MP method and the robustness of the reconstructions with respect to the measurement noise is discussed and proven. Specifically, the following results complement Sections \ref{sec:reconstruction} and \ref{sub_converse}, giving a rigorous proof for \eqref{eqn:01}-\eqref{eqn:conv},\eqref{eqn:cmp1}-\eqref{eqn:inclusions}.
\subsection{The theoretical limit} 
As already pointed out in \Cref{sec:reconstruction}, the theoretical limit proved in \Cref{prop_limits}, clearly highlights the role played by the Converse in MP imaging methods. Specifically, in the absence of noise, a MP imaging method reconstructs the anomaly $A$ with possibly some (or part) of its internal cavities that are not electrically connected to the boundary $\partial\Omega$.

\begin{theorem}
\label{prop_limits}
Let $A^{\dag}$ the pseudo-solution defined in \eqref{eqn:recon}. It result that
\begin{equation*}
    A\subseteq A^{\dag }\subseteq A^{\ast }. 
\end{equation*}
\end{theorem}
\begin{proof}
The Converse of MP \eqref{eqn:conv0}, in an equivalent form, is%
\begin{equation}
\overline{\Lambda }_{T}\leq \overline{\Lambda }_{A}\Rightarrow T\subseteq
A^{\ast }.  \label{eqn:06}
\end{equation}%
Equation (\ref{eqn:06}) together with definition (\ref{eqn:recon}), implies that $A^{\dag }\subseteq A^{\ast }$. Relationship $A\subseteq A^{\dag }$ is derived in \eqref{eqn:AAdag}.
\end{proof}

\subsection{The pseudosolution with noisy data}

In the presence of noise, the pseudosolution is%
\begin{equation}
A^{\dag }=\bigcup \left\{ T:\mathcal{P}_{A}^{\bm{\eta}}(f)-\mathcal{P}_{T}(f)\geq
0,\ \forall f\in X_\diamond\right\},  \label{eqn:11}
\end{equation}%
where $\mathcal{P}_{A}^{\bm{\eta}}$ is the noisy data as defined in \eqref{eqn:noisemodel}. Introducing $\mathcal{P}_{\varnothing }(f)$ as the measured power product (see \eqref{eqn:noisemes}) in absence of anomalies, i.e. for $A=\varnothing$, the pseudosolution can be conveniently recast as
\begin{equation}
\label{eqn:ps_noisy}
A^{\dag }=\bigcup \left\{ T:\Delta \mathcal{P}_{A,T}(f)+\mathcal{P}%
_{\varnothing }(f)\eta _{1}\xi _{1}+\xi _{2}\eta _{2}L\geq 0,\ \forall
f\in X_\diamond \right\} , 
\end{equation}
where $\Delta \mathcal{P}_{A,T}(f)=\left[ \mathcal{P}_{A}(f)-\mathcal{P}_{\varnothing }(f)\right] (1+\eta _{1}\xi _{1})-\left[ \mathcal{P}_{T}(f)-\mathcal{P}_{\varnothing }(f)\right]$.


The quantity $\Delta \mathcal{P}_{A,T}(f)$ can be made arbitrarily small, on a proper sequence $\left\{ f_{n}\right\} _{n\in \mathbb{N}}\subseteq X_\diamond$. This is due to the existence of a sequence $\{f_n\}_n$ of boundary potentials that localize $\gamma\abs{\nabla u}^2$ near the boundary of the domain $\Omega$, as proved in the next \Cref{thm_existence_sequence}.
\begin{proposition}
\label{thm_existence_sequence}
For any open bounded domain $D\Subset \Omega$ with Lipschitz boundary, there exists a sequence $%
\left\{ f_{k}\right\} _{k\in \mathbb{N}}\subseteq X_{\diamond }$ such that 
\begin{eqnarray}
\mathcal{P}_{\varnothing }(f_{n}) &=&1  \label{min1ass1_thm} \\
\lim_{n\rightarrow +\infty }\mathcal{P}_{D}(f_{n})-\mathcal{P}_{\varnothing
}(f_{n}) &=&0.  \label{min1ass2_thm} 
\end{eqnarray}
\end{proposition}

\begin{proof}
The proof is given in \Cref{Loc_bou_pot_appendix}.  
\end{proof}

\begin{proposition}
    \label{lem:min1}
Let $D$ be a set well contained in $\Omega$, there exists a sequence $\left\{ f_{n}\right\} _{n\in \mathbb{N}}\subseteq X_\diamond$ such that%
\[
\lim_{n\rightarrow +\infty }\Delta \mathcal{P}_{A,T}(f_{n})=0 .
\]
\end{proposition}

\begin{proof}
\Cref{thm_existence_sequence} assures the existence of a sequence $\{f_n\}_{n\in\N}$ satisfying \eqref{min1ass1_thm} and \eqref{min1ass2_thm}.

For any $\varepsilon >0$, let $n_{\varepsilon }$ be the smallest integer such that $\varepsilon /3>\mathcal{P}_{D}(f_{n})-\mathcal{P}%
_{\varnothing }(f_{n})\geq 0$, $\forall n>n_{\varepsilon }$. If $A,T\subseteq D$, it results that
\begin{eqnarray*}
\left\vert \Delta \mathcal{P}_{A,T}(f_{n})\right\vert &\leq &\left[ \mathcal{%
P}_{A}(f_{n})-\mathcal{P}_{\varnothing }(f_{n})\right] \left\vert 1+\eta
_{1}\xi _{1}\right\vert +\left[ \mathcal{P}_{T}(f_{n})-\mathcal{P}%
_{\varnothing }(f_{n})\right] \\
&\leq &\left[ \mathcal{P}_{D}(f_{n})-\mathcal{P}_{\varnothing }(f_{n})\right]
\left\vert 1+\eta _{1}\xi _{1}\right\vert +\left[ \mathcal{P}_{D}(f_{n})-%
\mathcal{P}_{\varnothing }(f_{n})\right] \\
&<&\frac{\varepsilon \left\vert 1+\eta _{1}\xi _{1}\right\vert +\varepsilon 
}{3} \\
&\leq &\varepsilon ,
\end{eqnarray*}%
for any $n>n_{\varepsilon }$.
\end{proof}




In the presence of noise, the pseudosolution defined in \eqref{eqn:recon} is the empty set, as proved in following Theorem.
\begin{theorem}\label{adagnullo}
Let $\xi_1$ and $\xi_2$ be two random variables uniformly distributed in $[-1,1]$. Under the assumptions of \Cref{lem:min1}, it results that $A^{\dag}=\varnothing$ with probability one, for noisy data.
\end{theorem}
\begin{proof}
From \eqref{eqn:ps_noisy} and \Cref{lem:min1}, for any $\varepsilon > 0$, it results that
\begin{equation*}
\begin{split}
A^{\dag } & \subseteq \bigcup \left\{ T:\Delta \mathcal{P}_{A,T}(f_n)+\mathcal{P}_{\varnothing }(f_n)\eta _{1}\xi _{1,n}+\xi _{2,n}\eta _{2}L\geq 0,\ \forall n\in\N \right\}\\
&\subseteq \bigcup \left\{ T:\Delta \mathcal{P}_{A,T}(f_n)+\eta _{1}\xi _{1,n}+\xi _{2,n}\eta _{2}L\geq 0,\ \forall n\geq n_\varepsilon \right\}\\
&\subseteq \bigcup \left\{ T:\varepsilon+\eta_{1}\xi _{1,n}+\xi _{2,n}\eta _{2}L\geq 0,\ \forall n\geq n_\varepsilon \right\},\\
\end{split}
\end{equation*}
where $\{ f_n \}_{n \in \N}$ and $n_{\varepsilon}$ are the sequence and the integer appearing the proof of \Cref{lem:min1}, respectively.

By choosing $0 < \varepsilon < \eta_1+\eta_2 L$, the random variable $\varepsilon+\eta_{1}\xi _{1,n}+\xi _{2,n}\eta _{2}L$ has a non vanishing probability to be negative. Therefore, given $T$, the probability that $\varepsilon+\eta_{1}\xi _{1,n}+\xi _{2,n}\eta _{2}L \ge 0$, for any $n > n_{\varepsilon}$, is zero. As a consequence, $A^{\dag}=\varnothing$ with probability one.
\end{proof}
\begin{remark} It is worth noting that the result of \Cref{adagnullo} holds also when the random variables $\xi_1$ and $\xi_2$ are not uniformly distributed. Indeed, \Cref{adagnullo} is valid as long as the probability density function of $\xi_1$ and $\xi_2$ are not zero in a neighborhood of the extremum $-1$.
\end{remark}

\subsection{Regularization}\label{sub_sec_regularization}
\Cref{adagnullo} shows the destructive impact of the noise on the monotonicity test, highlighting the need for a proper regularization of the reconstruction rule.

For the noise model of \eqref{eqn:noisemodel}, a regularized version of the monotonicity test \eqref{eqn:recrulenoise} is required. Specifically, ${A}^{\bm{\eta}}_{\bm{\eta}^{\ast }}$, the reconstruction obtained from noisy data and \emph{regularization parameters} $\bm{\eta}^{\ast }=(\eta_1^{\ast},\eta_2^{\ast})$, is defined as
\begin{equation}{A}^{\bm{\eta}}_{\bm{\eta}^{\ast }}=\bigcup \left\{ T:\frac{\mathcal{P}%
_{A}^{\bm{\eta}}(f)+\eta _{2}^{\ast }L}{1-\eta _{1}^{\ast }}-\mathcal{P}_{T}(f)\geq
0,\ \forall f\in X_\diamond\right\} .  \label{eq_01}
\end{equation}%
The set ${A}^{\bm{\eta}}_{\bm{\eta}^{\ast }}$ is a random set, because of the presence of the noisy data $\mathcal{P}%
_{A}^{\bm{\eta}}(f)$.

By plugging \eqref{eqn:noisemodel} into \eqref{eq_01}, it results that
\begin{equation}
\label{eqn:regsetstat}
{A}^{\bm{\eta}}_{\bm{\eta}^{\ast }}=\bigcup \left\{ T:\frac{\mathcal{P}_{A}(f)(1+\eta _{1}\xi _{1})+\xi _{2}\eta _{2}L+\eta _{2}^{\ast }L}{1-\eta
_{1}^{\ast }}-\mathcal{P}_{T}(f)\geq 0,\ \forall f\in X_\diamond\right\}.
\end{equation}
Starting from \eqref{eqn:regsetstat}, it is possible to define $A_{\bm{\eta}^{\ast }}^{\bm{\eta}}$ as
\begin{equation}\label{eqn:aee}
\mathbb A _{\bm{\eta}^{\ast }}^{\bm{\eta }}=\bigcup \left\{ T:\frac{\mathcal{%
P}_{A}(f)(1+\eta _{1})+\eta _{2}L+\eta _{2}^{\ast }L}{1-\eta _{1}^{\ast }}-%
\mathcal{P}_{T}(f)\geq 0,\ \forall f\in X_\diamond\right\} .     
\end{equation}
The set $\mathbb A_{\bm{\eta }^{\ast }}^{\bm{\eta }}$ is deterministic and depends on both the noise parameter $\bm{\eta}$ affecting the measurement, and the regularization parameters $\bm{\eta }^{\ast }$. The set $\mathbb A_{\bm{\eta}^{\ast }}^{\bm{\eta}}$ corresponds to ${A}^{\bm \eta}_{\bm{\eta}^{\ast}}$, when all realizations of the random variables $\xi _{1}$ and $\xi _{2}$ are equal to $1$.

In principle, the noise parameter $\bm{\eta}$ may be different from the regularization parameters $\bm{\eta^{\ast}}$.
However, when $\bm{\eta}=\bm{\eta}^{\ast}$ one achieves regularization. To this purpose, it is worth noting that
\begin{equation}
\mathbb A_{\mathbf{0}}^{\mathbf{0}}\subseteq A^{\bm \eta}_{\bm\eta}\subseteq \mathbb A_{\bm{\eta}}^{\bm{\eta}},  \label{eqn_09}
\end{equation}%
because%
\[
\mathcal{P}_{A}(f)\leq \frac{\mathcal{P}_{A}(f)(1+\eta _{1}\xi _{1})+\left(
1+\xi _{2}\right) \eta _{2}L}{1-\eta _{1}}\leq \frac{\mathcal{P}%
_{A}(f)(1+\eta _{1})+2\eta _{2}L}{1-\eta _{1}}. 
\]

\begin{remark}
Thanks to Proposition \ref{prop_limits} and the leftmost equality of (\ref%
{eqn_09}), it follows
\begin{align}
   & \mathbb A_{\mathbf{0}}^{\mathbf{0}}  = A^{\dag } = A_{\mathbf{0}}^{\mathbf{0}},\label{eqn:chain_reg}\\
    & A  \subseteq \mathbb A_{\mathbf{0}}^{\mathbf{0}} \subseteq A^{\ast },\\
    & A  \subseteq A_{\mathbf{0}}^{\mathbf{0}} \subseteq A^{\ast }. 
\end{align}
\end{remark}

\begin{remark}\label{rem:monot}
$\left\{\mathbb A_{\bm{\eta}}^{\bm{\eta}}\right\} _{\bm{\eta}}$ form a family of nonincreasing sets for $\eta _{1}$ decreasing $\left( 0<\eta_{1}<1\right) $\ and $\eta _{2}$ decreasing $\left( \eta _{2}>0\right) $.
\end{remark}

Let $\eta _{1,k}\downarrow 0^{+}$ and $\eta _{2,k}\downarrow 0^{+}$ be two monotonic sequences. Thanks to \Cref{rem:monot}, it results that $\left\{ \mathbb A_{\bm{\eta }_{k}}^{\bm{\eta }_{k}}\right\} _{k}$ is a sequence of monotonically non increasing sets and, therefore
\begin{equation}
\mathbb A_{\infty }=\lim_{k\rightarrow +\infty }\mathbb A_{\bm{\eta}_{k}}^{\mathbf{\bm\eta 
}_{k}}=\bigcap_{k=1}^{+\infty }\mathbb A_{\bm{\eta}_{k}}^{\bm{\eta}_{k}}.
\label{eqn:08}
\end{equation}

From (\ref{eqn_09}) and \eqref{eqn:chain_reg}, it results that%
\[
A^{\dag }\subseteq \mathbb A_{\infty } 
\]%
and, therefore, it remains to establish if the limiting set $\mathbb A_{\infty }$ is
larger than $A^{\dag }$ or not. To this purpose, the following
proposition holds.

\begin{proposition}
The closure of $A^{\dag }$ is equal to the closure of $\mathbb A_{\infty }$, i.e.%
\[
\overline{A^{\dag }}=\overline{\mathbb A_{\infty }}. 
\]
\end{proposition}

\begin{proof}
The interior of $\mathbb A_{\infty }\backslash A^{\dag }$ is the empty set. Indeed, let $T$
an open set contained in the interior of $\mathbb A_{\infty }\backslash A^{\dag }$.
It results that%
\begin{eqnarray}
\frac{\overline{\Lambda }_{A}(1+\eta _{1,k})+2\eta _{2,k}L}{1-\eta _{1,k}}-%
\overline{\Lambda }_{T} &\geq &0,\ \forall k\in\N  \label{eqn:09} \\
\overline{\Lambda }_{A}-\overline{\Lambda }_{T} &\ngeq &0.  \label{eqn:10}
\end{eqnarray}%
The first relationship comes from $T\subset \mathbb A_{\infty }$ and (\ref{eqn:08}),
the second from $T\nsubseteq A^{\dag }$ and the definition of $A^{\dag }$.
Passing in (\ref{eqn:09}) to the limit for $k\rightarrow +\infty $, it follows $%
\overline{\Lambda }_{A}-\overline{\Lambda }_{T}\geq 0$, that contradicts (%
\ref{eqn:10}).
\end{proof}

\begin{remark}
$A^{\dag }$ and $\mathbb A_{\infty }$ may differ only for boundary points.
\end{remark}
Summing up, given the regularization algorithm of \eqref{eqn:recrulenoise} and the family $\{ \mathbb A_{\bm{\eta}}^{\bm{\eta}} \}_{\bm \eta}$ defined in \eqref{eqn:aee}, it results that
\begin{itemize}
\item[(i)] the reconstruction is stable, in the sense that 
        \begin{equation*}
            A^{\dag}\subseteq {A}^{\bm{\eta}}_{\bm{\eta}}\subseteq \mathbb A_{\bm{\eta}}^{\bm{\eta}},
        \end{equation*}
i.e. there exists a deterministic upper and lower bounds that do not depend on the specific realization of the noise;
\item[(ii)] the reconstruction obtained via \eqref{eqn:recrulenoise} is convergent when the noise level approaches zero, that is, given two  monotone sequences $\eta _{1,k}\downarrow 0^{+}$ and $\eta _{2,k}\downarrow 0^{+}$
        \begin{equation*}
\overline{\lim_{k\to+\infty} A_{{\bm{\eta}_k}}^{\bm{\eta}_k}}  = \overline{\Aps}.
        \end{equation*}
Specifically, the closure of the limiting set is given by the closure of the pseudosolution for noise free data.
\end{itemize}

\subsection{Choice of the regularization parameter}
In the previous section it has been proved that choosing $\bm \eta = \bm \eta^{\ast}$ provides regularization. In this section it is evaluated the impact of a wrong estimate of the
regularization parameter $\bm{\eta}^{\ast }$ in the reconstruction
formula \eqref{eq_01}.

It is assumed that the regularization parameter $\bm{\eta}^{\ast }$ is a wrong estimates for the actual noise parameter $\bm{\eta}$ appearing in \eqref{eqn:noisemodel}. 
Specifically, it can be easily proved that 
\begin{eqnarray}
{A}_{\bm{\eta}}^{\bm{\eta}} &\subseteq &{A}^{\bm{\eta}}_{\bm{\eta}^{\ast
}},\text{ for }1>\eta _{1}^{\ast }\geq \eta _{1}>0\text{ and }1>\eta
_{2}^{\ast }\geq \eta _{2}>0  \label{eqn:13} \\
{A}^{\bm{\eta}}_{\bm{\eta}^{\ast }} &\subseteq &\mathbb {A}^{\bm{\eta}}_{\bm{\eta}%
},\text{ for }1>\eta _{1}\geq \eta _{1}^{\ast }>0\text{ and }1>\eta _{2}\geq
\eta _{2}^{\ast }>0.  \label{eqn:14}
\end{eqnarray}%
Therefore, if both noise parameters are overestimated, the reconstruction is
larger than ${A}_{\bm{\eta}}^{\bm{\eta}}$ (see equation (\ref{eqn:13})). On
the contrary, if both noise parameters are underestimated, the
reconstruction is smaller than ${A}^{\bm{\eta}}_{\bm{\eta}}$ (see equation (%
\ref{eqn:14})).

Moreover, it turns out that%
\begin{eqnarray*}
\lim_{\eta _{1}^{\ast },\eta _{2}^{\ast }\rightarrow 0^{+}}{A}_{\bm{\eta^{\ast}}}^{\bm{\eta}} &=&\varnothing \\
\lim_{\eta _{1}^{\ast },\eta _{2}^{\ast }\rightarrow 1^{-}}{A}^{\bm{\eta}}_{%
\bm{\eta^{\ast}}} &=&\bigcup_{T\in \mathcal{S}\left( \Omega \right) }T.
\end{eqnarray*}%
The set $\bigcup_{T\in \mathcal{S}\left( \Omega \right) }T$ represents the
largest anomaly that can be represented via $\mathcal{S}\left( \Omega
\right) $.

\section{Conclusions}
\label{sec:conclusions}
This paper provides the theoretical foundation for a recently introduced imaging method based on the Monotonicity Principle, to retrieve nonlinear anomalies in a linear background.

The main contribution of this work consisted in proving that (i) the imaging method is stable and robust with respect to the noise, (ii) the reconstruction approaches monotonically to a well-defined limit $\Aps$, as the noise level approaches to zero, and that (iii) the limit $\Aps$ satisfies $A \subseteq \Aps \subseteq A^{\ast}$, where $A^{\ast}$ is the outer boundary, that is the set $A$ plus all its cavities that are not connected by arch to $\partial \Omega$.

Results (i) and (ii) come directly from the Monotonicity Principle, while results (iii) requires to prove the Converse of the Monotonicity Principle, a theoretical result.

The results are proved for three wide classes of constitutive relationships covering the vast majority of cases encountered in applications.

\appendix
\section{Sharp estimates for perfectly conducting inclusions}
\label{sec:appa}
In this Appendix, the proof of Lemmas \ref{lem:moninf} and \ref{lem:ga} is provided, essential in the development of Subsection \ref{ano_infty}. With a little abuse of notation, the term \emph{perfectly conducting inclusion} denotes the region $A\subset \Omega$ such that $\gamma|_A=+\infty$, although $\gamma$ is a generic material property and not necessarily an electrical conductivity.

Consider the nonlinear material property $\gamma_A$ defined as in~\eqref{eqn:sigmaa}. It is reminded that $u_A$ is the solution of problem~\eqref{P1} with the material property $\gamma_A$ and the boundary data $f\in X_{\diamond}$.
For a perfect electrically conducting (PEC) inclusions $A\subset\Omega$, the material property $\gamma_A^{\infty}$ is defined in~\eqref{eqn:sigmainf} and the related scalar potential $u_A^{\infty}$ is the solution of~\eqref{u_infty}.

\begin{proof}[Proof of Lemma \ref{lem:moninf}]
$u_A\in H^1(\Omega)$ is the unique minimizer of \eqref{eqn:Ea}. As a consequence, 
\begin{equation*}
\begin{split}
    \langle \overline{\Lambda}_A(f),f \rangle&=
    \int_{\Omega}\int_0^{\abs{\nabla u_A(x)}}\gamma_A(x,\eta)\eta\,d\eta\,dx \\
    &\leq \int_{\Omega}\int_0^{\abs{\nabla u_A^{\infty}(x)}}\gamma_A(x,\eta)\eta\,d\eta\,dx\\
    &=\int_{\Omega\setminus A}\int_0^{\abs{\nabla u_A^{\infty}(x)}}\gamma_{BG}(x)\eta\,d\eta\,dx = \langle \overline{\Lambda}_A^{\infty}(f),f \rangle\qquad\forall f \in X_\diamond,
\end{split}
\end{equation*}
where it has been exploited that $|\nabla u_A^{\infty}|=0$ in $A$, the region where the material property is infinite.
\end{proof}

\begin{proof}[Proof of Lemma~\ref{lem:ga}]
The variational problem
\begin{equation}
    \label{minim_w}
    \min_{\substack{
    u\in H^1(\Omega\setminus A) \\ u_{|\partial\Omega}=0\\ u_{|\partial A}=u_{BG}-\overline{u_{BG}}} }\int_{\Omega\setminus  A}\gamma_{BG}(x)|\nabla u(x)|^2dx,
\end{equation}
is considered, where $u_{BG}$ is the solution of problem \eqref{u_BG} and $\overline{u_{BG}}$ is the average of $u_{BG}$ on $A$. 

If $w\in H^1(\Omega\setminus A)$ is the minimizer of \eqref{minim_w}, then it is the solution of
\begin{displaymath}
\begin{cases}
    \nabla\cdot(\gamma_{BG}(x)\nabla w(x))=0  & \text{in $\Omega\setminus A$} \\
    w(x)=0                                 & \text{on $\partial\Omega$}\\
    w(x)=u_{BG}(x)-\overline{u}_{BG}             & \text{on $\partial A$}.
\end{cases}
\end{displaymath}

By the inverse trace inequality, it is known that there exists a constant $C_1>0$ and $g\in H^1(\Omega\setminus A)$ with ${\rm Tr} (g)= 0$ on $\partial\Omega$ and ${\rm Tr} (g)= u_{BG}-\overline{u}_{BG}$ on $\partial A$ such that $||\nabla g||_{L^2(\Omega\setminus A)}\leq C_1 \lVert u_{BG}-\overline{u}_{BG} \rVert_{H^{1/2}(\partial A)}$.

Therefore, 
\begin{equation*}
\begin{split}
\int_{\Omega\setminus A} \gamma_{BG}(x)|\nabla w(x)|^2 dx & \leq \int_{\Omega\setminus A} \gamma_{BG}(x)|\nabla g(x)|^2 dx\leq \underline\gamma \lVert g \rVert_{L^2(\Omega\setminus A)}^2\\
& \leq\underline\gamma C_1 \lVert u_{BG}-\overline{u}_{BG} \rVert_{H^{1/2}(\partial A)}^2\leq \underline\gamma C_1 C_2 \lVert u_{BG}-\overline{u}_{BG} \rVert_{H^{1}(A)}^2\\
& \leq \underline\gamma C_1 C_2  \lVert \nabla u_{BG} \rVert_{L^2(A)}^2.
\end{split}
\end{equation*}
where in the first inequality the minimality of \eqref{minim_w} is used, in the third inequality the upper bound for $\gamma_{BG}$ is used, in the fourth inequality the inverse trace inequality on $\Omega\setminus A$ is used, in the fifth inequality, the classical trace inequality with constant $C_2>0$ is used, in the sixth inequality the Poincar\'e-Wirtinger inequality with constant $C_3>0$ is used.

Hence by setting $K=\underline\gamma C_1 C_2$, it holds
\begin{equation}\label{eqn:inpoi}
\int_{\Omega\setminus A} \gamma_{BG}(x)\abs{\nabla w(x)}^2\,dx\leq 
 K\lVert \nabla u_{BG} \rVert_{L^2(A)}^2.
\end{equation}

At this stage, it is observed that $w=u_{BG}-v$, where $v\in H^1(\Omega\setminus A)$ is the solution of
\begin{displaymath}
\begin{cases}
\nabla\cdot(\gamma_{BG}(x)\nabla v(x))=0  & \text{in $\Omega\setminus A$}\\
v(x)=f(x) & \text{on $\partial\Omega$}\\
v(x)=\overline{u}_{BG} & \text{on $\partial A$}.
\end{cases}
\end{displaymath}
It follows that
\begin{equation}\label{eqn:maj}
\begin{split}
    \int_{\Omega\setminus A}\gamma_{BG}(x)\abs{\nabla w(x)}^2\,dx&=\int_{\Omega\setminus A}\gamma_{BG}(x)\abs{\nabla u_{BG}(x)}^2\,dx\\
    &\quad+\int_{\Omega\setminus A}\gamma_{BG}(x)\abs{\nabla v(x)}^2\,dx\\
    &\quad-2\int_{\Omega\setminus A}\gamma_{BG}(x) \nabla u_{BG}(x)\cdot \nabla v(x)\,dx\\
    &=\int_{\Omega\setminus A}\gamma_{BG}(x)\abs{\nabla u_{BG}(x)}^2\,dx\\
    &\quad+\int_{\Omega\setminus A}\gamma_{BG}(x)\abs{\nabla v(x)}^2\,dx\\
    &\quad+2\int_{\Omega\setminus A}\gamma_{BG}(x) \nabla u_{BG}(x)\cdot \nabla w(x)\,dx.
\end{split}
\end{equation}
Furthermore, using the divergence Theorem, it follows
\begin{equation}
\label{div_ineq}
\begin{split}
    \int_{\Omega\setminus A}\gamma_{BG}(x)\nabla w(x)\cdot\nabla u_{BG}(x)\,dx&= 
    \int_{\partial\Omega}w(x) \gamma_{BG}(x)\partial_\nu u_{BG}(x)\,dx\\
    &\quad+\int_{\partial A}w(x) \gamma_{BG}(x)\partial_\nu u_{BG}(x)\,dx,
\end{split}
\end{equation}
where $\partial_\nu$ is the normal derivative along the outer direction w.r.t. $\Omega \setminus A$.

On the other hand, since $v$ is a constant on $\partial A$, $w=0$ on $\partial\Omega$, $A$ is well contained in $\Omega$ and 
\begin{displaymath}
    \int_{\partial A}\gamma_{BG}(x)\partial_\nu u_{BG}(x)=0,
\end{displaymath}
then \eqref{div_ineq} becomes
\begin{equation}\label{eqn:eqint}
\begin{split}
\int_{\Omega\setminus A}\gamma_{BG}(x)\nabla w(x)\cdot\nabla u_{BG}(x)\,dx&=    \int_{\partial A}w(x) \gamma_{BG}(x)\partial_\nu u_{BG}(x)\,dx\\
&=\int_{\partial A}u_{BG}(x) \gamma_{BG}(x)\partial_\nu u_{BG}(x)\,dx\\
&=-\int_A \gamma_{BG}(x) \abs{\nabla u_{BG}(x)}^2\, dx.
\end{split}
\end{equation}

Combining~\eqref{eqn:inpoi},~\eqref{eqn:maj} and~\eqref{eqn:eqint}, the following inequality is obtained
\begin{equation}
\label{K+1}
\int_{\Omega\setminus A}\gamma_{BG}(x)\abs{\nabla v(x)}^2\,dx-\int_{\Omega\setminus A}\gamma_{BG}(x)\abs{\nabla u_{BG}(x)}^2\,dx\leq  K\lVert \nabla u_{BG} \rVert_{L^2(A)}^2.
\end{equation}
Since $u_A^\infty$ is solution of \eqref{u_infty},  it is also the unique minimizer of
\begin{displaymath}
 \min_{\substack{ u\in H^1(\Omega) \\ \nabla u=0 \text{ in }A \\  u|_{\partial\Omega}=f}}  \  \int_{\Omega\setminus A}\gamma_{BG}(x)\abs{\nabla u(x)}^2\,dx,
\end{displaymath}
then it results that
\begin{equation}
    \label{minim_ineq}
    \int_{\Omega\setminus A}\gamma_{BG}(x)\abs{\nabla v(x)}^2\,dx\geq \int_{\Omega\setminus A}\gamma_{BG}(x)\abs{\nabla u_A(x)}^2\,dx.
\end{equation}
Therefore, by combining \eqref{K+1} and \eqref{minim_ineq}, it turns out that 
\begin{equation*}
\begin{split}
(K_1+1)\int_{A}\gamma_{BG}(x)\abs{\nabla u_{BG}(x)}^2\,dx&\geq
    \int_{\Omega\setminus A}\gamma_{BG}(x)\abs{\nabla u_A(x)}^2\,dx \\
    &\quad - \int_{\Omega}\gamma_{BG}(x)\abs{\nabla u_{BG}(x)}^2\,dx\\
    &=\langle \Lambda_A^{\infty}(f)-\Lambda_{BG}(f),f\rangle,
\end{split}
\end{equation*}
and the conclusion follows by the fact that $\gamma_{BG}\in L^\infty_+(\Omega)$.
\end{proof}

\section{Sharp estimates for perfectly insulating inclusions}
\label{sec:appb}
In this Appendix, the proof of Lemmas \ref{lem:monnul} and \ref{lem:locpei} are provided, essential in the development of Subsection \ref{ano_zero}. Analogously to \Cref{sec:appa}, \emph{perfectly insulating inclusions} indicate the region $A\subset\Omega$ for which $\gamma|_A=0$, although $\gamma$ can be a magnetic permeability or a dielectric permittivity as well as an electrical conductivity.

Consider a perfect electrical insulting (PEI) inclusion $A\subset \Omega$, with material property given by $\gamma_A^0$ defined as in~\eqref{eqn:sigma0},
and $u_A^0$ being the solution of~\eqref{u_zero}.

\begin{proof}[Proof of Lemma \ref{lem:monnul}]
The solution $u_A^0$ of \eqref{u_zero} is the unique minimizer of
\begin{equation}
    \label{minim_prob_zero}
\min_{\substack{u\in H^1(\Omega\setminus A) \\ u|_{\partial\Omega}=f }}  \int_{\Omega\setminus A}\int_0^{\abs{\nabla u(x)}}\gamma_{BG}(x)\eta\,d\eta\,dx.
\end{equation}
Furthermore, it is observed that
\begin{equation*}
\begin{split}
    \langle \overline{\Lambda}_A(f),f \rangle&=
    \int_{\Omega}\int_0^{\abs{\nabla u_A(x)}}\gamma_A(x,\eta)\eta \,d\eta \,dx \\
    &\geq \int_{\Omega\setminus A}\int_0^{\abs{\nabla u_A(x)}}\gamma_A(x,\eta)\eta\,d\eta\,dx\\
    &=\int_{\Omega\setminus A}\int_0^{\abs{\nabla u_A(x)}}\gamma_{BG}(x)\eta\,d\eta\,dx \\
    &\ge \int_{\Omega\setminus A}\int_0^{\abs{\nabla u_A^0(x)}}\gamma_{BG}(x)\eta\,d\eta\,dx \\
    &= \langle \overline{\Lambda}_A^0(f),f \rangle,
\end{split}
\end{equation*}
where in the first line the definition of the average DtN operator $\overline{\Lambda}_A$ is used, in the second line the fact the integral restricts on a smaller domain is used, in the third line the definition of $\gamma_A^0$ as in~\eqref{eqn:sigma0} is used, in the fourth line the minimality of problem \eqref{minim_prob_zero} is used and in the fifth line the definition of the average DtN operator $\overline{\Lambda}_A^0$ is used.
\end{proof}

\begin{proof}[Proof of Lemma~\ref{lem:locpei}]
This proof is an adaptation of that of~\cite[Lemma~5.3]{art:Ga20}.

Let $w_n=v_n|_{\Omega\setminus S_1}-u_n|_{\Omega\setminus S_1}\in H^1(\Omega\setminus S_1)$ be the difference (in $\Omega\setminus S_1$) between the solutions in the absence and in the presence of a perfect insulating anomaly in $S_1$, respectively, i.e. $w_n$ is the solution of
\begin{equation*}
    \begin{cases}
\nabla\cdot\left(\gamma_{BG}(x)\nabla w_n(x) \right)=0 & \text{in $\Omega\setminus S_1$}\\
 w_n(x)=0 & \text{on $\partial\Omega$}\\
\gamma_{BG}(x)\partial_\nu w_n(x)=\gamma_{BG}(x)\partial_\nu v_n(x) & \text{on $\partial S_1$},
    \end{cases}
\end{equation*}
where $v_n$ solves
\begin{equation*}
    \begin{cases}
\nabla\cdot(\gamma_{BG}(x)\nabla v_n(x))=0 & \text{in $\Omega$} \\
v_n(x)=f_n(x) & \text{on $\partial\Omega$}
    \end{cases}
\end{equation*}
and $u_n$, restricted to $\Omega\setminus S_1$, solves problem \eqref{u_zero*}.

Moreover, $w_n$ solves the following variational problem\begin{equation}    \label{minim_w_n}\min_{\substack{u\in H^1(\Omega\setminus S_1) \\ u_{|\partial\Omega}=0\\ \gamma_{BG}\partial_\nu u_{|\partial S_1}=\gamma_{BG}\partial_\nu {v_n}_{|\partial S_1}} }\int_{\Omega\setminus  S_1}\gamma_{BG}(x)|\nabla u(x)|^2dx.\end{equation}

Furthermore, recalling that $S_2\subset \subset \Omega\setminus S_1^*$ and $S_1^*$ is the complement in $\overline{\Omega}$ of the union of those relatively open set $V$ contained in $\Omega \setminus\overline S_1$ and connected to $\partial \Omega$, then it immediately follows that there exists a relatively open set $U\subset\overline{\Omega}$ intersecting the boundary $\partial\Omega$ and such that $S_1^*\subset \Omega\setminus\overline{U}$ and $S_2 \subset U$.

It follows
\begin{equation*}
\begin{split}
\overline\gamma\norm{w}_{H^{1/2}(\partial S_1)}&\norm{\nabla w_n}_{L^2(\Omega\setminus S_1)}\leq C_1\overline\gamma \norm{w}_{H^{1}(\Omega\setminus S_1)}\norm{\nabla w_n}_{L^2(\Omega\setminus S_1)}\\
&\le C_1 C_2 \overline\gamma\norm{\nabla w_n}_{L^2(\Omega\setminus S_1)}^2\\
&\le C_1C_2 \int_{\Omega\setminus  S_1}\gamma_{BG}(x)|\nabla w_n(x)|^2dx\\
&=C_1C_2 \langle w_n,\gamma_{BG}\partial_{\nu} w_n\rangle_{\partial(\Omega\setminus S_1)}\\
&=C_1C_2 \langle w_n,\gamma_{BG}\partial_{\nu} w_n\rangle_{\partial  S_1}\\
&\leq C_1C_2 \norm{w_n}_{H^{1/2}(\partial S_1)}\lVert \gamma_{BG}\partial_\nu v_n \rVert_{H^{-1/2}(\partial S_1)}\\
&\leq C_1 C_2 \norm{w_n}_{H^{1/2}(\partial S_1)}\lVert \gamma_{BG}\partial_\nu v_n \rVert_{H^{-1/2}(\partial S_1 \cup (\partial\Omega\setminus \partial U))}\\
&\leq C_1C_2\norm{w_n}_{H^{1/2}(\partial S_1)}(\lVert \gamma_{BG}\nabla v_n\rVert_{L^2(\Omega\setminus(S_1 \cup U))}+\lVert \nabla\cdot(\gamma_{BG}\nabla v_n) \rVert_{L^2(\Omega\setminus(S_1 \cup U))})\\
&=C_1C_2 \norm{w_n}_{H^{1/2}(\partial S_1)}\lVert \gamma_{BG}\nabla v_n \rVert_{L^2(\Omega\setminus(S_1 \cup U))},
\end{split}
\end{equation*}
where in the first line the trace inequality with constant $C_1$ is used,
in the second line the generalized Poincar\'e inequality with constant $C_2$ is used,
in the third line the lower bound for $\gamma_{BG}$ is used,
in the fourth line the divergence theorem is used, where $\langle \cdot,\cdot\rangle$ is the duality pairing $H^{1/2}(\partial(\Omega\setminus S_1)\times H^{-1/2}(\partial(\Omega\setminus S_1)$,
in the fifth line it is recognized that $\gamma_{BG}\partial_n v_n=\gamma_{BG}\partial_n w_n$ on $\partial S_1$, 
in the sixth line the definition of operatorial norm in $H^{-1/2}(\partial S_1)$ is used,
in the seventh line it is exploited the fact that the integral increases on bigger sets, 
in the eighth line the fact that the trace of the normal component is a bounded map from $H_{div}(\Omega\setminus(S_1\cup U))$ to $H^{-1/2}(\partial S_1 \cup (\partial\Omega\setminus \partial U))$ is used and, 
in the nineth line the fact that $\nabla\cdot(\gamma_{BG}\nabla v_n)=0$ in $\Omega\setminus(S_1 \cup U)$ is used. Refer to \Cref{fig:ta9} for the geometric details.

\begin{figure}[htp]
\centering
\includegraphics[width=.4\textwidth]{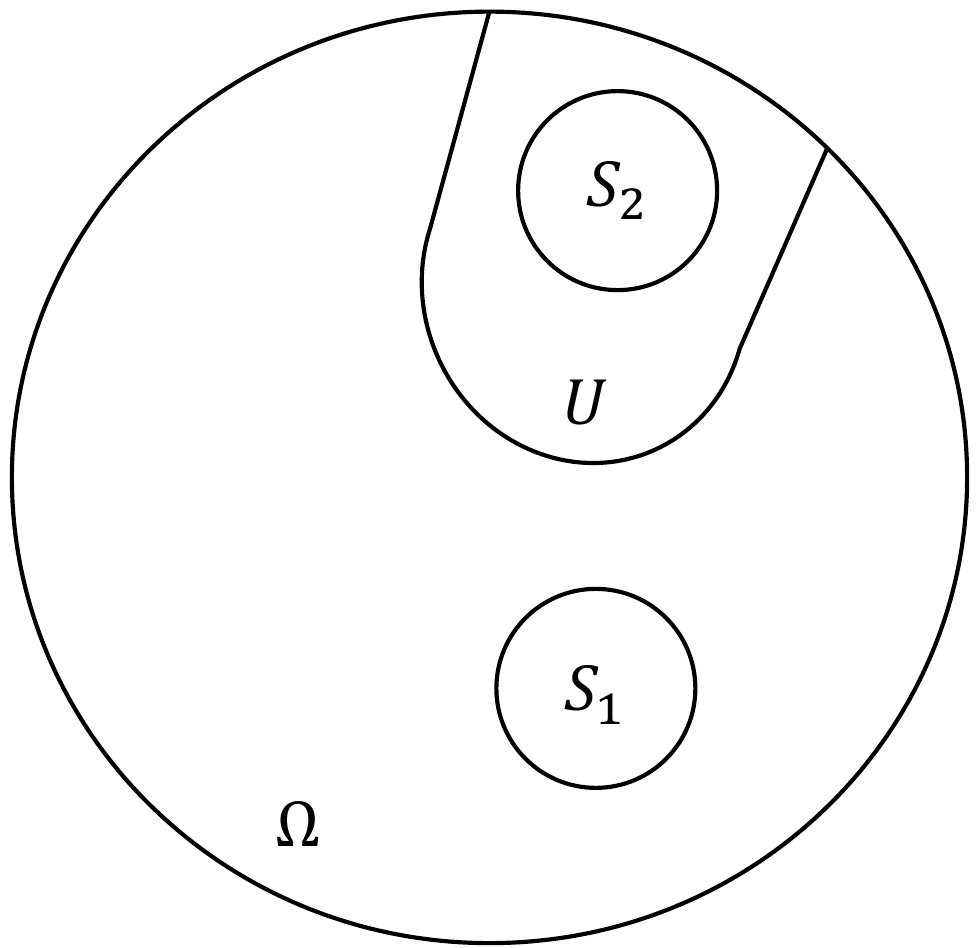}
\caption{Geometric relationships between sets $\Omega$, $U$, $S_1$ and $S_2$.}
\label{fig:ta9}
\end{figure}
Hence, by setting $K:=\frac{C_1C_2}{\overline{\gamma}}$, it results
\begin{equation}
    \label{eqn:abin}
\norm{\nabla w_n}_{L^2(\Omega\setminus S_1)}\leq K \lVert \gamma_{BG}\nabla v_n \rVert_{L^2(\Omega\setminus(S_1 \cup U))}.
\end{equation}
By \eqref{variante_loc}, there exists a sequence of boundary potentials $\{f_n\}_{n\in\mathbb{N}}\subset X_{\diamond}$ such that the solutions $\{v_n\}_{n\in\mathbb{N}}$ of
\begin{equation*}
    \begin{cases}
        \nabla\cdot\left(\gamma_{BG}(x) \nabla v_n(x)\right)=0 & \text{in $\Omega$}, \\
        v_n(x)=f_n(x) & \text{on $\partial\Omega$},
    \end{cases}
\end{equation*}
fulfill
\begin{equation}\label{eqn:locdim}
\lim_{n\to+\infty}\int_{\Omega\setminus U}\abs{\nabla v_n(x)}^2\,dx=0\quad\text{and}\quad\lim_{n\to+\infty}\int_{S_2}\abs{\nabla v_n(x)}^2\,dx=+\infty.
\end{equation}
From \eqref{eqn:abin} and~\eqref{eqn:locdim}, it follows that $||\nabla w_n||_{L^2(\Omega\setminus S_1)}$ converges to zero when the localized potentials are applied on the boundary. Furthermore, recalling that $u_n=v_n|_{\Omega\setminus S_1}-w_n$, it follows that the limiting behaviour of $\nabla v_n$ coincides with the limiting behaviour of $\nabla u_n$ on $\Omega\setminus S_1$. So
\begin{equation}\label{senza_S1}
\lim_{n\to+\infty}\int_{\Omega\setminus (S_1 \cup U)}\abs{\nabla u_n(x)}^2\,dx=0\quad\text{and}\quad\lim_{n\to+\infty}\int_{S_2}\abs{\nabla u_n(x)}^2\,dx=+\infty.
\end{equation}
In order to conclude, it remains to investigate the asymptotic behavior on $S_1$. Since $u_n$ tends to zero in $H^1(\Omega\setminus(S_1\cup U))$, by applying the trace theorem, it follows that $\norm{u_n}_{H^{1/2}(\partial S_1)}$ tends to 0. Since $u_n$ in $S_1$ solves \eqref{u_zero*_cascata}, then it solves
\begin{equation}
    \label{minim_u_n}
\min_{\substack{u\in H^1(\Omega\setminus S_1) \\ u_{|\partial S_1}=g_n}}\int_{\Omega\setminus  S_1}|\nabla u(x)|^2dx,
\end{equation}
being $g_n$ the restriction to $\partial S_1$ of the solution $u_n$ of problem \eqref{u_zero*}.

By the inverse trace inequality, there exists a constant $C>0$ and $h\in H^1( S_1)$ with ${\rm Tr} (h)= g_n$ on $\partial S_1$ such that $||\nabla h||_{L^2(S_1)}\leq C_1 \lVert u_n \rVert_{H^{1/2}(\partial S_1)}$. Therefore, by using the minimality of \eqref{minim_u_n}, it follows that
\begin{displaymath}
\int_{S_1} \gamma_{BG}(x)\abs{\nabla u_n(x)}^2\,dx\leq C_\infty
   \int_{S_1} \abs{\nabla h(x)}^2\,dx\leq C_\infty C_1 \lVert u_n \rVert_{H^{1/2}(\partial S_1)}.
\end{displaymath}
Since the last term tends to zero, it results that
\begin{equation}
\label{solo_S1}
\lim_{n\to+\infty}\int_{S_1} \gamma_{BG}(x)\abs{\nabla u_n(x)}^2\,dx=0.
\end{equation}
The conclusion follows by joining \eqref{senza_S1} and \eqref{solo_S1}.
\end{proof}

\section{Localized boundary potential}
\label{Loc_bou_pot_appendix}
In this appendix, the existence of depleting potentials along a particular sequence of boundary data is proven by the use of attenuating exponential solutions of the Laplace equation. This proof relies on the geometric circumstance that it is possible to find a hyperplane through a point of the boundary such that there exists a positive measure region close to the boundary. This assumption is valid and the proof is given in \Cref{existence_H_thm}. 

This appendix is divided into two parts. In the first one, the existence of the hyperplane mentioned above is proved with geometric and analytical methods; in the second part, it is proved a key inequality, required to establish the main result stated in \Cref{thm_existence_sequence}.

\subsection{The hyperplane existence}
For the sake of convenience, here the definition of convex hull is recalled.
\begin{definition}
    Let $E\subseteq\R^N$, the convex hull of $E$ is the smallest convex set containing $E$, i.e.
    \[
    \textrm{co}(E):=\left\{\sum_{h=1}^k\lambda_hx_h \ : \ x_h\in E,\ x_h\geq 0\  \forall\  h=1,...,k,\ \sum_{h=1}^k\lambda_h=1\right\}.
    \]
\end{definition}
To pursue our aim, it is necessary to prove that the convex hull is monotone with respect to the inclusion.
\begin{lemma}
Let $\Omega$, $D$ open bounded domains with Lipschitz boundary. It holds that $\overline{\textrm{co(D)}}=\textrm{co}(\overline{D})\subset \textrm{co}(\mathring{\Omega})=\mathring{\overline{\textrm{co}(\Omega)}}$, that means $\textrm{co}(D)\Subset \textrm{co}(\Omega)$.
\end{lemma}
\begin{proof}
It is trivial to observe that if $E$ is an open (closed) set, then also $\textrm{co}(E)$ is an open (closed) set, whence the first and the last equality. 

    Now, set $\varepsilon:=\dist (\overline{D},\R^N\setminus\Omega)>0$ and consider $y\in co(\overline{D})$. By definition, there exist $x_1,...,x_k\in\overline{D}$ and $\lambda_1,...,\lambda_k\geq 0$ with $\sum_{h=1}^k\lambda_h=1$, such that
    \[
    y=\sum_{h=1}^k\lambda_hx_h.
    \]
    If $v\in\R^N$ such that $||v||<\varepsilon$ is considered, then $y+v=\sum_{h=1}^k\lambda_h(x_h+v)\in co(\Omega)$. It is observed that $B(x_h,||v||)\subseteq B(x_h,\varepsilon)\subset\Omega$, and hence this implies that $x_h+v\in\Omega$.

    This means that $B(y,\varepsilon)\subset \textrm{co}(\Omega)$, that means $\dist(\textrm{co}(\overline{D}), \R^N\setminus \textrm{co}(\Omega))\geq \varepsilon$.
\end{proof}
The second important key Lemma states that the boundary of the convex hull touches the boundary of the prescribed domain.
\begin{lemma} \label{lemma_touch}
Let $\Omega$ be an open bounded domain with Lipschitz boundary, it holds that
    \[
\partial(co(\Omega))\cap\partial\Omega\neq \emptyset.
    \]
\end{lemma}
\begin{proof}
Since $\textrm {co}(\Omega)\supseteq\Omega$, it is easily seen that $\partial(\textrm{co}(\Omega))\cap\Omega= \emptyset$. By contradiction, if the assert is not true, it turns out that
\[
\lambda=\dist (\Omega,\R^N\setminus \textrm{co}(\Omega))=\dist(\partial\Omega,\partial (\textrm{co}(\Omega)))>0.
\]
Consider $x_0\in\Omega\subset \textrm{co}(\Omega)$ and $R>0$ such that $co(\Omega)\subset B(x_0,R)$. Meanwhile, define
\[    C_\mu:=\{x_0+\mu (x-x_0)\ : \ x\in \textrm{co}(\Omega), \ (1-\mu)R<\lambda\},\]
where $\mu>1-\frac{\lambda}{R}$ is fixed. For any $x\in \textrm{co}(\Omega)$, the relation $x-x_0-\mu(x-x_0)=(x-x_0)(1-\mu)$ presents the left factor lower than $R$ in norm, while the right factor is lower than $\lambda/{R}$. Therefore, it is easily seen that $C_\mu$ is convex and $\dist(C_\mu,\R^N\setminus co(\Omega))\leq (1-\mu)R<\lambda$, hence
 $C_\mu\supseteq\Omega$ that gives a contradiction.
\end{proof}

Thanks to the previous Lemmata, the existence of a suitable hyperplane can be proved.
\begin{proposition}\label{existence_H_thm}
    Let $\Omega$, $D$ be open bounded domains of $\R^N$ with Lipschitz boundary such that $\overline{D}\subset  \mathring{\Omega}=\Omega$. Then, there exists $x_0\in\partial\Omega$, a direction $\nu$, an hyperplane $\Pi=\{x\in\R^N \ : \ (x-x_0)\cdot\nu=0\}$ and a constant $\delta>0$ such that $D\subseteq\{x\in\R^N \ : \ (x-x_0)\cdot\nu\ge \delta\}$.
\end{proposition}
\begin{proof}
By \Cref{lemma_touch}, it results that
$    \delta=\dist(D,\R^N\setminus\Omega)\leq \dist(\textrm{co}(D),\partial \textrm{co}(\Omega))$; therefore, for any $x_0\in \partial\Omega\cap \partial (co(\Omega))$, it follows $\dist(x_0,\partial(\textrm{co(D)}))\geq \delta>0$. 

By setting $x_1\in\partial(co(D))$ the projection of $x_0$ on $\overline{\textrm{co}(D)}$, it results that
\[
\dist(x_0,x_1)=\min_{x\in\overline {co(D)}}\dist(x_0,x)>\delta
\]
and $\nu=\frac{x_1-x_0}{||x_1-x_0||}$.

Finally, the sought hyperplane is defined as
\[
\Pi=\{y\in\R^N \ : \ y\cdot \nu=x_0\cdot \nu\}.
\]
It is worth noting that $x_0\in\Pi$ and $(x-x_0)\cdot\nu\geq (x-x_0)\cdot\nu=||x_1-x_0||$.
\end{proof}

\subsection{The depleting potentials}

Before proving the main result, a preliminary Lemma is given for the treatment of inclusions with vanishing material property.

\begin{lemma}\label{in_ins_lemma}
Let $\Lambda_{BG}$ and $\Lambda_D^{0}$  be the DtN operators related to the material properties $\gamma_{BG}(x) \in L_+^\infty(\Omega)$ and 
    \begin{equation*}
        \gamma_D^0(x)=\begin{cases}
            0 & \text{in } D \\
            \gamma_{BG}(x) & \text{in } \Omega\setminus D.
        \end{cases}
    \end{equation*}
     Assuming $\gamma_{BG}(x)$ piecewise analytic and condition~\eqref{eqn:nass2}, there exists a positive constant $K$ such that
\begin{equation*}
0\leq\langle \Lambda_{BG}(f) - \Lambda_D^{0}(f), f\rangle\leq K\int_D\abs{\nabla u_{BG}(x)}^2\,dx
\qquad\forall f \in X_\diamond,
\end{equation*}
\end{lemma}
\begin{proof}
Let $w$ be the solution of the following problem
\begin{equation*}
    \begin{cases}
        \nabla \cdot (\gamma_{BG}\nabla w)=0 & \text{in }\Omega\setminus D \\
        w=0 & \text{on }\partial\Omega\\
        \partial_{\nu}w=\partial_{\nu}u_{\varnothing} & \text{on }\partial D,
    \end{cases}
\end{equation*}
the following estimation holds (see also proof of Lemma 4.10)
\begin{equation*}
\begin{split}
\overline\gamma\norm{w}_{H^{1/2}(\partial S_1)}&\norm{\nabla w_n}_{L^2(\Omega\setminus S_1)}\leq C_1\overline\gamma \norm{w}_{H^{1}(\Omega\setminus S_1)}\norm{\nabla w_n}_{L^2(\Omega\setminus S_1)}\\
&\le C_1 C_2 \overline\gamma\norm{\nabla w_n}_{L^2(\Omega\setminus S_1)}^2\\
&\le C_1C_2 \int_{\Omega\setminus  S_1}\gamma_{BG}(x)|\nabla w_n(x)|^2dx\\
&=C_1C_2 \langle w_n,\gamma_{BG}\partial_{\nu} w_n\rangle_{\partial(\Omega\setminus S_1)}\\
&=C_1C_2 \langle w_n,\gamma_{BG}\partial_{\nu} w_n\rangle_{\partial  S_1}\\
&\leq C_1C_2 \norm{w_n}_{H^{1/2}(\partial S_1)}\lVert \gamma_{BG}\partial_\nu w_n \rVert_{H^{-1/2}(\partial S_1)},
\end{split}
\end{equation*}
i.e,
\begin{equation*}
    \norm{\nabla w}_{L^2(\Omega\setminus D)}\leq k_1 \lVert \gamma_{BG}\partial_\nu u_{\varnothing} \rVert_{H^{-1/2}(\partial D)}.
\end{equation*}
Furthermore, by the inverse trace inequality for Neumann data
\begin{equation*}
    \lVert \gamma_{BG}\partial_\nu u_{\varnothing} \rVert_{H^{-1/2}(\partial D)}\leq k_2 \norm{\nabla u_{\varnothing}}_{L^2(D)}.
\end{equation*}
Hence
\begin{equation}\label{eqn:init1}
    \norm{\nabla w}_{L^2(\Omega\setminus D)}\leq k_0 \norm{\nabla u_{\varnothing}}_{L^2(D)}.
\end{equation}
On the other hand, by observing that $w=u_{\varnothing}|_{\Omega\setminus D}-u_0$, where $u_0$ solves
\begin{equation*}
    \begin{cases}
        \nabla\cdot(\gamma_{BG}(x)\nabla u_0(x))=0 & \text{in} \Omega\setminus D \\
        u_0=f & \text{on } \partial\Omega\\
        \gamma_{BG}\partial_{\nu}u_0=0 & \text{on } \partial D,
    \end{cases}
\end{equation*}
it follows
\begin{equation}\label{eqn:init2}
\begin{split}
    k_3 \norm{\nabla w}^2_{L^2(\Omega\setminus D)}&\geq\int_{\Omega\setminus D} \gamma_{BG}(x)\abs{w(x)}^2\,dx\\
    &=\int_{\Omega\setminus D}\gamma_{BG}(x) \abs{u_{\varnothing}}^2\,dx+\int_{\Omega\setminus D}\gamma_{BG}(x) \abs{u_0}^2\,dx\\
    &\: -2\int_{\Omega\setminus D}\gamma_{BG}(x)\nabla u_{\varnothing}(x)\cdot \nabla u_0(x)\,dx \\
    &=\int_{\Omega\setminus D}\gamma_{BG}(x) \abs{\nabla u_{\varnothing}}^2\,dx-\int_{\Omega\setminus D}\gamma_{BG}(x) \abs{\nabla u_0}^2\,dx \\
    &=\langle (\Lambda_{\varnothing}-\Lambda_D^0)(f),f\rangle - \int_{D}\gamma_{BG}(x) \abs{\nabla u_{\varnothing}}^2\,dx
\end{split}
\end{equation}
where in the fourth line it is exploited the fact that
\begin{equation*}
    \int_{\Omega\setminus D}\gamma_{BG}(x) \abs{\nabla u_0}^2\,dx=\int_{\Omega\setminus D}\gamma_{BG}(x)\nabla u_{\varnothing}(x)\cdot \nabla u_0(x)\,dx.
\end{equation*}
Combining \eqref{eqn:init1} and \eqref{eqn:init2} the claim follows.
\end{proof}

\begin{proof}[Proof of \Cref{thm_existence_sequence}]
The following proof exploits some results from \cite{kohn1984determining,kohn1985determining}.
By \Cref{existence_H_thm}, it follows that there exists a point $x_{0}\in
\partial \Omega $, a hyperplane $H$ of dimension $n-1$ through $x_{0}$ and a
constant $\delta >0$ such $D\subseteq \left\{ x_{0}+\xi _{n}\mathbf{\hat{%
\imath}}_{n}\in 
\mathbb{R}
^{n}|\xi _{n}>\delta \right\} $, where $\mathbf{\hat{\imath}}_{1},\ldots ,%
\mathbf{\hat{\imath}}_{n}$ are the unit vectors of a reference system having
the origin in $x_{0}$ and oriented so that $\mathbf{\hat{\imath}}_{n}$ is
normal to $H$ and directed toward $D$. 

Since $\gamma_{BG}$ is piecewise analytic, it is possible to assume that, apart from a zero measure set, $x_0$ admits a neighborhood $U$ such that (i) $\gamma_{BG}\in C^{\infty}(U)$, and (ii) $D\subset\Omega\setminus\overline{U}$. From Lemmata 1, 2 and 3 of \cite{kohn1984determining}, there exists a sequence $\{f_n\}_n\subset C^{\infty}(\partial\Omega)$ and a constant $C>0$, such that $\norm{f_n}_{H^{1/2}(\partial\Omega)}=1$ and
\begin{equation}\label{eqn_kohn}
\int_{U}\abs{\mathbf{s}_{\emptyset}}^2\,dx\geq C  \quad \text{and} \quad \lim_{n\to+\infty}\int_{\Omega\setminus \overline{U}} \abs{\mathbf{s}_{\emptyset}}^2\,dx=0,
\end{equation}
where $\mathbf{s}_{\emptyset}=-\nabla u_{\emptyset}$ and $u_{\emptyset}$ solves the background problem, i.e. the original problem \eqref{P1} in the absence of anomalies ($A=\emptyset$) and for a Dirichlet data given by $f_n$.

Hence, for region $D$ it turns out that
\[
\mathbb G_{D}\left( f_{n}\right) 
:=\int_{D} \gamma _{BG}(x)\left\vert \mathbf{s}_{\varnothing }\right\vert^{2}\text{d}x
\leq \overline{\gamma}\int_{D}\left\vert \mathbf{s}_{\varnothing }\right\vert^{2}\text{d}x
\leq \overline{\gamma}\int_{\Omega\setminus\overline{U}}\left\vert \mathbf{s}_{\varnothing }\right\vert^{2}\text{d}x
\]%
while for region $\Omega$ it turns out that
\[
\mathbb G_{\Omega}\left( f_{n}\right) 
:=\int_{\Omega} \gamma _{BG}(x)\left\vert \mathbf{s}_{\varnothing }\right\vert^{2}\text{d}x
\geq \underline{\gamma}\int_{\Omega}\left\vert \mathbf{s}_{\varnothing }\right\vert^{2}\text{d}x
\geq \underline{\gamma}\int_{U}\left\vert \mathbf{s}_{\varnothing }\right\vert^{2}\text{d}x.
\]%
Therefore, by means of \eqref{eqn_kohn}, it results that%
\begin{equation}\label{eqn_limb}
   \lim_{n\rightarrow +\infty }\frac{{\mathbb G}_{D}\left( f_{n}\right) }{{\mathbb G}_{\Omega
}\left( f_{n}\right) }=0, 
\end{equation}
while condition (\ref{min1ass1_thm}) can be imposed by properly scaling the boundary data $f_n$.

In the remainder of this proof, it is proven that the difference $\mathcal{P}_{D}(f_n)-\mathcal{P}_{\varnothing}(f)$ is controlled by $P_{D}$, which together with \eqref{eqn_limb} gives \eqref{min1ass2_thm}. In order to do that, a different treatment is required for the various classes of nonlinearity treated in this work.

\proofpart{1}{$\gamma_{NL}>\gamma_{BG}$ and bounded.}
By replacing $\mathcal{P}_A(f)-\mathcal{P}_T(f)$ with $\mathcal{P}_D(f)-\mathcal{P}_{\varnothing}(f)$ in (4.11), it holds
\begin{equation*}
\begin{split}
0\leq\mathcal{P}_D(f)-\mathcal{P}_{\varnothing}(f) &\leq \int_{\Omega}\int_0^{{\left\vert \mathbf{s}_{\varnothing }\right\vert}}\left(\gamma_D(x,\eta)-\gamma_{BG}\right)\eta\,d\eta\,dx\\
&=\int_{D}\int_0^{{\left\vert \mathbf{s}_{\varnothing }\right\vert}}\left(\gamma_{NL}(x,\eta)-\gamma_{BG}\right)\eta\,d\eta\,dx \\
&\leq \int_{D}\int_0^{{\left\vert \mathbf{s}_{\varnothing }\right\vert }}\left(\overline{\gamma}-\gamma_{BG}\right)\eta\,d\eta\,dx \\
& \leq \alpha {\mathbb G}_D(f),
\end{split}   
\end{equation*}
where
\begin{equation*}
    \alpha=\frac{\overline{\gamma}-\gamma_{BG}^l}{\gamma_{BG}^l}
\end{equation*}
and $\gamma_{BG}^l$ is the lower bound for the background material property.

Hence,
\begin{equation*}
0\leq \lim_{{n\to+\infty}} \left[ \mathcal{P}_D(f_n)-\mathcal{P}_{\varnothing}(f_n) \right] \leq \alpha  \lim_{{n\to+\infty}}  {\mathbb G}_D(f_n)=0. 
\end{equation*}

\proofpart{2}{$\gamma_{NL}<\gamma_{BG}$ and bounded.}
Let $\gamma_D^I$ the material property defined as in expression (4.21) by replacing $A$ with $D$ and $\mathcal{P}_D^I(f)=\langle \overline{\Lambda}_D^I(f),f\rangle$. By the MP it follows $\mathcal{P}_D\geq\mathcal{P}_D^I$.
Furthermore, by replacing $\mathcal{P}_T$ with $\mathcal{P}_D^I$ in (4.37), it holds
\begin{equation*}
\begin{split}
    0\leq\mathcal{P}_{\varnothing}(f)-\mathcal{P}_D^I(f) &\leq\frac{1}{2}\int_D\frac{\gamma_{BG}(x)}{\underline{\gamma}}(\gamma_{BG}(x)-\underline{\gamma})\abs{\mathbf{s}%
_{\varnothing }}^2\,dx \\
    & \leq \alpha \mathbb G_D(f),
\end{split}  
\end{equation*}
where 
\begin{equation*}
    \alpha=\frac{1}{2}\frac{\gamma_{BG}^u-\underline{\gamma}}{\underline{\gamma}}>0
\end{equation*}
and $\gamma_{BG}^u$ upper bound to the background material property.

Hence,
\begin{equation*}
0\leq  \lim_{{n\to+\infty}} \left[ \mathcal{P}_{\varnothing}(f_n)-\mathcal{P}_D(f_n) \right] \leq \alpha  \lim_{{n\to+\infty}} \mathbb G_D(f_n)=0. 
\end{equation*}

\proofpart{3}{$\gamma_{NL}>\gamma_{BG}$ and unbounded.}
Let $\mathcal{P}_D^{\infty}(f)=\langle \overline{\Lambda}_D^{\infty}(f),f\rangle$, i.e. the power product when the nonlinear material filling $D$ is replaced by a material with $\gamma=+\infty$. Combining Lemma 4.6 and Lemma 4.7, it holds
\begin{equation*}
    0\leq \mathcal{P}_D(f)-\mathcal{P}_{\varnothing}(f) \leq \mathcal{P}_D^{\infty}(f)-\mathcal{P}_{\varnothing}(f)\leq \alpha {\mathbb G}_{D}(f).
\end{equation*}

Hence,
\begin{equation*}
0\leq \lim_{n\to +\infty} \left[ \mathcal{P}_D(f_n)-\mathcal{P}_{\varnothing}(f_n) \right] \leq \alpha  \lim_{{n\to+\infty}}  {\mathbb G}_D(f_k)=0. 
\end{equation*}

\proofpart{4}{$\gamma_{NL}<\gamma_{BG}$ and possibly vanishing.}
Combining Lemma 4.9 and \Cref{in_ins_lemma} it holds
\begin{equation*}
    0\leq \mathcal{P}_{\varnothing}(f)-\mathcal{P}_D(f) \leq \mathcal{P}_{\varnothing}(f)-\mathcal{P}_D^{0}(f)\leq \alpha {\mathbb G}_{D}(f).
\end{equation*}

Hence,
\begin{equation*}
0\leq  \lim_{{n\to+\infty}} \left[ \mathcal{P}_{\varnothing}(f_n) - \mathcal{P}_D(f_n)\right] \leq \alpha  \lim_{{n\to+\infty}} {\mathbb G}_D(f_n)=0. 
\end{equation*}
\end{proof}
\begin{remark}
For a constant background, i.e. $\gamma_{BG}(x)=\gamma_{BG}$ constant, the depleting potentials can be found explicitly.    

The change of variable between the original reference system and the ad-hoc
reference system centered in $x_0$ (see \Cref{existence_H_thm}) is described by $x=x_{0}+\sum_{j=1}^{{N}}\xi _{j}\mathbf{\hat{%
\imath}}_{j}$.

Let the boundary data $f_n$\ be defined as the restriction to $\partial
\Omega $ of a plane wave exponentially attenuating in the $\mathbf{\hat{%
\imath}}_{{N}}$ direction:%
\[
f_{n}\left( x\right) =a_{n}e^{-\xi _{{N}}/\delta _{n}}\sin \left( \beta
_{1}\xi _{1}+\ldots +\beta _{{N}-1}\xi _{{N}-1}\right) \ \forall x\in \partial
\Omega , 
\]%
where $\delta _{n}=\delta /2^{n}$ and the $\beta _{j}$s are real and
constrained to satisfy%
\[
\beta _{1}^{2}+\ldots +\beta _{{N}-1}^{2}=\frac{1}{\delta _{{N}}^{2}}. 
\]

Since the background material property is constant, the solution of
(1.1) for $A=\varnothing $ is $u_{\varnothing }=a_{n}e^{-\xi _{{N}}/\delta
_{n}}\sin \left( \beta _{1}\xi _{1}+\ldots +\beta _{N-1}\xi _{N-1}\right) $.
The squared norm of $\mathbf{s}_{\varnothing }=-\nabla
u_{\varnothing }$ is%
\[
\left\vert \mathbf{s}_{\varnothing }\right\vert ^{2}=\frac{a_{n}^{2}}{\delta
_{n}^{2}}e^{-2\xi _{{N}}/\delta _{n}}. 
\]%
Hence, for the anomaly $D$
\[
{\mathbb G}_{D}\left( f_{n}\right) =\gamma _{BG}\int_{D}\left\vert \mathbf{s}%
_{\varnothing }\right\vert ^{2}\text{d}x=\gamma_{BG}\frac{a_{n}^{2}}{\delta
_{n}^{2}}\int_{D}e^{-2\xi _{{N}}/\delta _{n}}\text{d}x\leq \gamma_{BG}\frac{%
a_{n}^{2}}{\delta _{n}^{2}}e^{-2\delta /\delta _{n}}\int_{D}\text{d}x, 
\]%
while for $\Omega $%
\[
{\mathbb G}_{\Omega }\left( f_{n}\right) =\gamma _{BG}\int_{\Omega }\left\vert \mathbf{%
s}_{\varnothing }\right\vert ^{2}\text{d}x=\gamma_{BG}\frac{a_{n}^{2}}{%
\delta _{n}^{2}}\int_{\Omega }e^{-2\xi _{{N}}/\delta _{n}}\text{d}x\geq \gamma
_{BG}\frac{a_{n}^{2}}{\delta _{n}^{2}}e^{-2\frac{\delta }{2}/\delta
_{n}}\int_{\Omega _{\frac{\delta }{2}}}\text{d}x, 
\]%
where $\Omega _{\frac{\delta }{2}}=\Omega \cap \left\{ x_{0}+\xi _{{N}}\mathbf{%
\hat{\imath}}_{{N}}\in 
\mathbb{R}
^{{N}}|\xi _{{N}}<\delta /2\right\} $. Therefore, since%
\[
\frac{{\mathbb G}_{D}\left( f_{n}\right) }{{\mathbb G}_{\Omega }\left( f_{n}\right) }\leq
e^{-\delta /\delta _{n}}\frac{m\left( D\right) }{m\left( \Omega _{\frac{%
\delta }{2}}\right) }, 
\]%
it results that%
\[
   \lim_{n\rightarrow +\infty }\frac{{\mathbb G}_{D}\left( f_{n}\right) }{{\mathbb G}_{\Omega
}\left( f_{n}\right) }=0, 
\]
as for equation \eqref{eqn_limb}.
The coefficient $a_{n}$ is chosen to impose (\ref{min1ass1_thm}):%
\[
a_{n}^{2}=\frac{\delta _{n}^{2}}{\gamma _{BG}\int_{\Omega }e^{-2\xi
_{{N}}/\delta _{n}}\text{d}x}. 
\]
\end{remark}

\section*{Acknowledgments}
This work has been partially supported by the Italian Ministry of University and Research (projects n. 20229M52AS and n. 2022Y53F3X, PRIN 2022), and by GNAMPA of INdAM.

The Authors would like to thank the Journal for the high professionalism of the peer-review process, and the Reviewers for their valuable comments.

\section*{Authorship contribution statement}

{\bf A. Corbo Esposito, A. Tamburrino}: Conceptualization, Methodology, Formal analysis, Writing, Supervision.

{\bf L. Faella, V. Mottola, G. Piscitelli, R. Prakash}: Conceptualization, Methodology, Formal analysis, Writing. 


\section*{Data availability statement}
No new data were created or analysed in this study.

\bibliographystyle
{siam}
\bibliography{references}
\end{document}